\documentclass[12pt]{amsart}
\usepackage{amscd, amsfonts, amssymb, amsmath}
\usepackage{fullpage}
\usepackage{latexsym}
\usepackage{verbatim}
\usepackage{enumerate}
\usepackage[all,cmtip]{xy}
\usepackage[colorlinks=true,citecolor=red,linkcolor=blue]{hyperref}

\newcommand{\ds}{\displaystyle}

\newcommand\ra{\rightarrow}
\newcommand{\iso}{\cong}

\numberwithin{equation}{section}

\newtheorem{thm}[equation]{Theorem}

\newtheorem{lem}[equation]{Lemma}
\newtheorem{cor}[equation]{Corollary}
\newtheorem{prop}[equation]{Proposition}
\newtheorem{qn}[equation]{Question}

\theoremstyle{definition}
\newtheorem{defn}[equation]{Definition}
\newtheorem{ex}[equation]{Example}

\theoremstyle{remark}
\newtheorem{rem}[equation]{Remark}
\theoremstyle{remark}

\newtheorem{assumption}[equation]{Assumption}

\newcommand{\ovl}{\overline}

\date{July 25, 2022}

\title[Algebraic groups and $G$-complete reducibility]
{Algebraic groups and $G$-complete reducibility: \\ a geometric approach}

\author[B.\ Martin]{Benjamin Martin}
\address
{Department of Mathematics,
University of Aberdeen,
King's College,
Fraser Noble Building,
Aberdeen AB24 3UE,
United Kingdom}
\email{B.Martin@abdn.ac.uk}

\begin{document}

\begin{abstract}
 The notion of a {\em $G$-completely reducible} subgroup is important in the study of algebraic groups and their subgroup structure.  It generalizes the usual idea of complete reducibility from representation theory: a subgroup $H$ of a general linear group $G= {\rm GL}_n(k)$ is $G$-completely reducible if and only if the inclusion map $i\colon H\ra {\rm GL}_n(k)$ is a completely reducible representation of $H$.  In these notes I  give an introduction to the theory of complete reducibility and its applications, and explain an approach to the subject using geometric invariant theory.
\end{abstract}

\maketitle

These notes are based on a series of six lectures delivered at the International Workshop on ``Algorithmic problems in group theory, and related areas'', held at the Oasis Summer Camp near Novosibirsk from July 26 to August 4, 2016.\smallskip\\

\noindent {\bf Expected background:} The reader is assumed to have a basic understanding of the algebraic geometry of affine varieties over an algebraically closed field: the material in \cite[Ch.\ 1]{hum} or \cite[Ch.\ AG]{borel} is more than sufficient.  I do {\bf not} assume any knowledge of linear algebraic groups; there is a brief review in Section~\ref{sec:1}.  Nonetheless the notes will be followed more easily by someone who has some familiarity with the subject.  For an introduction, I recommend one (or more!) of the excellent books of Humphreys \cite{hum}, Springer \cite{springer} and Borel \cite{borel}.\smallskip\\

\noindent {\bf Changes from the original notes:} The original version of these notes and the associated exercise sheets can be found at the conference website:\\ \url{http://www.math.nsc.ru/conference/isc/2016/main_eng.htm}.

  In this version I have added some references, corrected some mistakes and incorporated some material that was originally covered in the exercises.  I have reordered some text to make the exposition more coherent and expanded some of the arguments.\smallskip\\

\noindent {\bf Acknowledgements:} I am grateful to the workshop organisers Evgeny Vdovin, Alexey Galt, Fedor Dudkin, Maria Zvezdina, Andrey Mamontov, and Alexey Staroletov for their hospitality, and for permission to make these notes publicly available.  I also used these notes as a basis for some lectures at a Spring School on Complete Reducibility at Bochum University in April 2018, and I thank the organisers Gerhard R\"ohrle, Falk Bannuscher, Maike Gruchot and Alastair Litterick for their hospitality.  I'm also grateful to the workshop and spring school participants for their comments and for pointing out various mistakes.

\newpage
\centerline{{\sc Contents}}

\medskip
\begin{enumerate}
 \item[1.] {\bf Motivation and review of algebraic groups}
 \item[2.] {\bf $G$-complete reducibility}
 \item[3.] {\bf Geometric invariant theory}
 \item[4.] {\bf A geometric criterion for $G$-complete reducibility}
 \item[5.] {\bf Optimal destabilising parabolic subgroups}
 \item[6.] {\bf Other topics}
\end{enumerate}

\medskip

\section{Motivation and review of algebraic groups}
\label{sec:1}

My aim in these notes is to discuss the theory of $G$-completely reducible subgroups of a reductive linear algebraic group $G$, and describe an approach using geometric invariant theory.  In particular, I give a reasonably complete and self-contained explanation of two key results---the characterisation of $G$-completely reducible subgroups in terms of closed orbits (Theorem~\ref{thm:Gcr_crit}) and the construction of the Kempf-Rousseau-Hesselink optimal destabilising parabolic subgroup (Theorem~\ref{thm:opt})---introducing the necessary ideas from geometric invariant theory along the way.  Much of this geometric approach is based on work of Michael Bate, Benjamin Martin and Gerhard R\"ohrle \cite{BMR}, but many of the ideas are originally due to Roger Richardson \cite{rich4}.

Algebraic groups are to algebraic geometry what Lie groups are to differential geometry, and they appear in many areas of group theory, algebraic geometry and number theory.  Simple (and, more generally, reductive) algebraic groups arise as automorphism groups of interesting structures such as Lie algebras or Jordan algebras; they are the source of finite groups of Lie type and they have applications to spherical buildings.

The main idea of $G$-complete reducibility is to generalise the definition of a completely reducible representation from representation theory---which involves subgroups of a general linear group ${\rm GL}_n(k)$---to subgroups of an arbitrary reductive algebraic group $G$.  One can then check which results from representation theory work in this more general setting.  For instance, Clifford's Theorem still holds (see Theorem~\ref{thm:Clifford} below).

\medskip
\noindent {\bf Some applications:}
\begin{enumerate}
 \item The subgroup structure of reductive algebraic groups.  (Given a subgroup $H$ of $G$, either $H$ is $G$-completely reducible or it isn't!  We can say useful things in either case.)
 \item Simple groups of Lie type.
 \item Spherical buildings.
 \item Geometric invariant theory.
\end{enumerate}

\begin{ex}
\label{ex:GLn_rep}
 Let $V$ be an $n$-dimensional vector space over $k$ and let $\rho\colon H\ra {\rm GL}(V)$ be a representation of a linear algebraic group $H$.  We want to know whether $\rho$ is completely reducible.  Only the image of $\rho$ is important, so we might as well assume that $H\leq {\rm GL}(V)$ and $\rho$ is inclusion.
 
 Let $W$ be a subspace of $V$ and let $m= {\rm dim}(W)$.  Choose a basis $e_1,\ldots, e_m$ for $W$, then extend this to a basis for $V$ by adding vectors $e_{m+1},\ldots, e_n$.  We can identify ${\rm GL}(V)$ with ${\rm GL}_n(k)$ via this choice of basis, so we get an inclusion $\psi\colon H\ra {\rm GL}_n(k)$.  We see that $W$ is $H$-stable if and only if $\psi(H)\leq P$, where $P$ is the group of block upper triangular matrices with an $m\times m$ block followed by an $(n-m)\times (n-m)$ block down the leading diagonal.
 
 Now suppose $W$ is $H$-stable.  Then $W$ has an $H$-stable complement if and only if we can choose $e_{m+1},\ldots, e_n$ above in such a way that they span an $H$-stable subspace of $V$.  Now $e_{m+1},\ldots, e_n$ span an $H$-stable subspace of $V$ if and only if $\psi(H)\leq L$, where $L$ is the group of block diagonal matrices with an $m\times m$ block followed by an $(n-m)\times (n-m)$ block down the leading diagonal.
 
 When we change the basis for a vector space $V$, the matrix of a linear transformation of $V$ changes via conjugation by the change of basis matrix.  We have proved the following: if $H$ is a subgroup of ${\rm GL}_n(k)$ then $H$ stabilises an $m$-dimensional subspace $W$ of $k^n$ if and only if $H$ is ${\rm GL}_n(k)$-conjugate to a subgroup of $P$, and $H$ stabilises both an $m$-dimensional subspace and a complement to that subspace if and only if $H$ is ${\rm GL}_n(k)$-conjugate to a subgroup of $L$.  This gives a purely intrinsic description of complete reducibility in terms of subgroups of ${\rm GL}_n(k)$, without involving the vector space $k^n$.  It is this idea that we want to generalise.  We return to this situation in Example~\ref{ex:DLn_SLn_crit}.
\end{ex}

\subsection{Review of algebraic groups}
\label{sec:alggps}

{\ }\medskip\\
\noindent {\bf Some notation:} Let $k$ be an algebraically closed field.  Let $X$ be an affine variety over $k$.  We write $k[X]$ for the co-ordinate ring of $X$, and if $X$ is irreducible then we write $k(X)$ for the function field of $X$.  Given $x\in X$, we denote the tangent space to $X$ at $x$ by $T_xX$.  Given a morphism $\phi\colon X\ra Y$ of affine varieties, we denote by $d_x\phi\colon T_xX\ra T_{\phi(x)}Y$ the derivative of $\phi$ at $x$.

We recall the derivative criterion for separability.  Let $\phi\colon X\ra Y$ be a dominant morphism of irreducible affine varieties.  The comorphism $\phi^*\colon k[Y]\ra k[X]$ is injective, so it gives rise to an embedding of $k(Y)$ in $k(X)$.  We say that $\phi$ is {\em separable} if $k(X)/k(Y)$ is a separable field extension.  If $x\in X$ such that $x$ and $\phi(x)$ are smooth points and $d_x\phi$ is surjective then $\phi$ is separable.  Conversely, if $\phi$ is separable then there is a nonempty open subset $U$ of $X$ such that for all $x\in U$, $x$ and $\phi(x)$ are smooth points and $d_x\phi$ is surjective.\smallskip\\

\noindent {\bf Algebraic groups:} By ``algebraic group'' we mean ``linear algebraic group''.  Let $k$ be an algebraically closed field and let ${\rm SL}_n(k)$ be the group of $n\times n$ matrices of determinant 1 with entries from $k$.  An {\em algebraic group} is a closed subgroup of ${\rm SL}_n(k)$ for some $n$: that is, a subgroup of ${\rm SL}_n(k)$ that is the zero set of a set of polynomials in the matrix entries.  For instance, the special orthogonal group ${\rm SO}_n(k)$ is the subgroup of ${\rm SL}_n(k)$ given by the condition $AA^T= I$, and this condition can be expressed as $n^2$ polynomial equations in the $n^2$ matrix entries of $A$.  Note that ${\rm GL}_n(k)$---the group of invertible $n\times n$ matrices with entries from $k$---can be viewed as a closed subgroup of ${\rm SL}_{n+1}(k)$; it follows easily that a subgroup of ${\rm GL}_n(k)$ that is the zero set of a set of polynomials in the matrix entries is an algebraic group.   Moreover, ${\rm SL}_n(k)$ itself is given---as a subset of $M_n(k)$, the set of all $n\times n$ matrices over $k$---by the condition ${\rm det}(A)= 1$, which can be expressed in terms of polynomial equations in the matrix entries.  The additive group $(k,+)$ and the multiplicative group $(k^*,\cdot)$ of the field are algebraic.

An equivalent definition (but not obviously so!): An algebraic group is an affine variety with a group structure such that group multiplication and inversion are morphisms of varieties.

\medskip
\noindent {\bf The Zariski topology, subgroups and homomorphisms:} An algebraic group $H$ is an affine variety, so it carries the Zariski topology.  As $H$ acts transitively on itself by left multiplication, it is smooth, so its irreducible components and connected components coincide and these components all have the same dimension.  We denote by $H^0$ the unique connected component that contains the identity.  By a subgroup of $H$ we mean a closed subgroup unless otherwise stated; such a subgroup is an algebraic group in its own right.  If $N$ is a normal subgroup of $H$ then $H/N$ is also an algebraic group (this takes some work to prove).  A homomorphism of algebraic groups is assumed to be a morphism of varieties.  If $\phi\colon H_1\ra H_2$ is a homomorphism of algebraic groups then $\phi(H_1)$ is closed in $H_2$.  An epimorphism of algebraic groups with finite kernel is called an {\em isogeny}.  A product of two algebraic groups is an algebraic group.  By a {\em representation}\footnote{Often these are called {\em rational representations} in the literature: the adjective ``rational'' serves to emphasise that $\rho$ is a morphism of varieties.} of $H$ we mean a homomorphism $\rho$ from $H$ to ${\rm GL}_n(k)$ for some $n\in {\mathbb N}$ (or, equivalently, to ${\rm GL}(V)$ for some finite-dimensional vector space $V$ over $k$).  We often write $g\cdot v$ for $\rho(g)(v)$.  We define irreducibility and complete reducibility of representations in the usual way.

\medskip
\noindent {\bf The Lie algebra:} As for Lie groups, we can associate to an algebraic group $H$ a Lie algebra ${\mathfrak h}$ over $k$; ${\mathfrak h}$ is the tangent space $T_1(H)$ at the identity.  The action of $H$ on itself by conjugation gives rise to a representation ${\rm Ad}\colon H\ra {\rm GL}({\mathfrak h})$, called the {\em adjoint representation}.  For instance, we can identify the Lie algebra $\mathfrak{gl}_n(k)$ of ${\rm GL}_n(k)$ with $M_n(k)$ (with Lie bracket given by $[A,B]:= AB- BA$), and we have $g\cdot A= gAg^{-1}$ for $g\in {\rm GL}_n(k)$ and $A\in M_n(k)$.  The Lie algebra $\mathfrak{sl}_n(k)$ of ${\rm SL}_n(k)$ is the subalgebra of $M_n(k)$ consisting of the traceless matrices, and the adjoint action is also given by conjugation.

The close correspondence between Lie groups and their Lie algebras can falter for algebraic groups.  For instance, if ${\rm char}(k)= p> 0$ and $p$ divides $n$ then the centre $Z({\rm SL}_n(k))$ is finite, but the Lie algebra centre ${\mathfrak z}(\mathfrak{sl}_n(k))$ is 1-dimensional because scalar multiples of the identity matrix are traceless.  In the language of representations, ${\mathfrak z}(\mathfrak{sl}_n(k))$ is an ${\rm Ad}$-stable subspace of $\mathfrak{sl}_n(k)$; it is in fact easy to show that $\mathfrak{sl}_n(k)$ is not completely reducible.  In contrast, $\mathfrak{sl}_n(k)$ {\bf is} completely reducible if $p$ does not divide $n$.

\medskip
\noindent {\bf Unipotent and semisimple elements:} Let $i\colon H\ra {\rm GL}_n(k)$ be an embedding\footnote{By ``embedding'' we mean ``closed embedding''.} of algebraic groups (at least one such $i$ exists for any given $H$, by definition!)  We say that $h\in H$ is {\em unipotent} if $i(h)$ is conjugate to an upper unitriangular matrix (upper triangular with 1s on the diagonal), and we say that $h$ is {\em semisimple} if $i(h)$ is conjugate to a diagonal matrix (i.e., is diagonalisable).  This does not depend on the choice of embedding $i$.  There exist unique $h_{\rm s}, h_{\rm u}\in H$ such that $h_{\rm s}$ is semisimple, $h_{\rm u}$ is unipotent, $h= h_{\rm s}h_{\rm u}$ and $h_{\rm s}$ and $h_{\rm u}$ commute (this is the {\em Jordan decomposition} of $h$).  Uniqueness implies that if $M\leq H$ and $h$ belongs to $M$ then $h_{\rm s}$ and $h_{\rm u}$ also belong to $M$.  If $\phi\colon H_1\ra H_2$ is a homomorphism of algebraic groups and $h\in H_1$ then $\phi(h_{\rm s})= \phi(h)_{\rm s}$ and $\phi(h_{\rm u})= \phi(h)_{\rm u}$.  The group $H$ is {\em unipotent} if every element of $H$ is unipotent.  ({\bf Warning:} A semisimple group is {\bf not} a group consisting only of semisimple elements.)  Subgroups and quotients of unipotent groups are unipotent, and conversely, if $1\ra N\ra H\ra Q\ra 1$ is a short exact sequence of algebraic groups then $H$ is unipotent if $N$ and $Q$ are (these facts follow easily from the Jordan decomposition).  In characteristic 0, every nontrivial unipotent element has infinite order, every element of finite order is semisimple and every unipotent subgroup is connected.  In characteristic $p> 0$, an element is unipotent if and only if its order is a power of $p$, and an element of finite order is semisimple if and only if its order is coprime to $p$.

\medskip
\noindent {\bf The unipotent radical and reductive groups:} The {\em unipotent radical} $R_u(H)$ of $H$ is the unique largest connected normal unipotent subgroup of $H$.  We say that $H$ is {\em reductive} if $R_u(H)= 1$ (this is one of the most important and most useless definitions in the theory of algebraic groups!).  In general, $H/R_u(H)$ is reductive.  Since $R_u(H)= R_u(H^0)$, $H$ is reductive if and only if $H^0$ is reductive.

\medskip
\noindent {\bf Maximal tori and Borel subgroups}: A {\em torus} is an algebraic group that is isomorphic to $(k^*)^m$ for some $m$.  Any algebraic group contains a maximal torus $T$, and $T$ is unique up to conjugacy.  We define the {\em rank} of $H$ to be the dimension of a maximal torus of $H$.  A quotient of a torus is a torus.  A {\em Borel subgroup} $B$ of $H$ is a maximal connected solvable subgroup of $H$; these are also unique up to conjugacy.  (It is clear from the definitions that a conjugate of a maximal torus is a maximal torus, and a conjugate of a Borel subgroup is a Borel subgroup.)  For instance, if $H= {\rm SL}_n(k)$ then the subgroup $D_n= D_n(k)$ of diagonal matrices is a maximal torus, and the subgroup $B_n= B_n(k)$ of upper triangular matrices is a Borel subgroup; $R_u(B_n)$ is $U_n= U_n(k)$, the subgroup of upper unitriangular matrices.

\medskip
\noindent {\bf Characters and weights:} Let $X(H)$ denote the set of homomorphisms $\chi\colon H\ra k^*$; we call elements of $X(H)$ {\em characters}.  The set $X(H)$ is an abelian group under pointwise multiplication.  We use additive notation for $X(H)$.  The homomorphisms from $k^*$ to $k^*$ are precisely the maps of the form $a\mapsto a^n$ for some $n\in {\mathbb Z}$, so $X(k^*)\iso {\mathbb Z}$.  More generally, if $S$ is a torus of dimension $m$ then $X(S)$ is a free abelian group on $m$ generators.

Now let $\rho\colon H\ra {\rm GL}(V)$ be a representation and let $S$ be a torus of $G$.  We say that $0\neq v\in V$ is a {\em weight vector} of $V$ if there is a function $\chi\colon S\ra k^*$ such that $h\cdot v= \chi(h)v$ for all $h\in S$.  The function $\chi$ is uniquely determined by $v$ and $\chi$ is a character of $S$; we call $\chi$ a {\em weight} of $V$ with respect to $S$.  We define $\Phi_S(V)$ to be the set of weights.  If $\chi\in \Phi_S(V)$ then the set $V_\chi:= \{v\in V\mid h\cdot v= \chi(h)v\ \mbox{for all $h\in S$}\}$ is a subspace of $V$, called the {\em weight space} corresponding to $\chi$.  We have $\ds V= \bigoplus_{\chi\in \Phi_S(V)} V_\chi$---this follows from the classical result that commuting diagonalisable matrices can be simultaneously diagonalised.  So, given any $v\in V$, we have a unique decomposition $\ds v= \sum_{\chi\in \Phi_S(V)} v_\chi$, where each $v_\chi\in V_\chi$.  We define ${\rm supp}_S(v)= \{\chi\in \Phi_S(V)\mid v_\chi\neq 0\}$, and we call this set the {\em support} of $v$ (with respect to $S$): so we have $\ds v= \sum_{{\rm supp}_S(V)} v_\chi$.

The notion of a {\em cocharacter} will be crucial for us.  We discuss cocharacters in Section~\ref{sec:GIT}.

\medskip
\noindent {\bf Linearly reductive groups:} An algebraic group is {\em linearly reductive} if every representation of it is completely reducible.  The above argument shows that any torus is linearly reductive.  More generally, any algebraic group consisting entirely of semisimple elements is linearly reductive.

\medskip
\noindent {\bf Parabolic subgroups and Levi subgroups}: Let $H$ be connected and reductive.  A {\em parabolic subgroup} of $H$ is a subgroup $P$ of $H$ that contains a Borel subgroup of $H$ (such a subgroup is automatically closed); in particular, $H$ is a parabolic subgroup of $H$.  A {\em Levi subgroup} of $P$ is a maximal reductive subgroup $L$ of $P$.  Levi subgroups of $P$ exist; they are not unique, but they are unique up to $R_u(P)$-conjugacy (it is clear that a $P$-conjugate of a Levi subgroup of $P$ is a Levi subgroup of $P$).  Moreover, $P$ is isomorphic to the semidirect product $L\ltimes R_u(P)$; we denote by $c_L$ the canonical projection from $P$ to $L\cong P/R_u(P)$.  Parabolic subgroups and their Levi subgroups are connected.  We have $C_H(P)= Z(H)$ and $N_H(P)= P$ (here $C_H(\cdot)$ and $N_H(\cdot)$ denote the centraliser and normaliser, respectively), and $C_H(L)^0= Z(L)^0$ is a torus.  (In particular, $Z(H)^0$ is a torus as $H$ is a parabolic subgroup of $H$ with unique Levi subgroup $H$.)  Any conjugate of a parabolic subgroup is a parabolic subgroup, and $H$ has only finitely many conjugacy classes of parabolic subgroups.  If $P_1$ and $P_2$ are parabolic subgroups of $H$ then $P_1\cap P_2$ contains a maximal torus of $H$.  Later we will give another characterisation of parabolic subgroups and Levi subgroups.  We abuse notation and speak of Levi subgroups of $H$; by this we mean Levi subgroups of parabolic subgroups of $H$.

For example, let $H= {\rm SL}_n(k)$ (the description for ${\rm GL}_n(k)$ is completely analogous).  Fix ${\mathbf n}= (n_1,\ldots, n_r)\in {\mathbb N}^r$ such that $n_1+\cdots \cdots+ n_r= n$.  The subgroup $P_{\mathbf n}$ of block upper triangular matrices with diagonal block sizes $n_1,\ldots, n_r$ down the leading diagonal is a parabolic subgroup of ${\rm SL}_n(k)$; conversely, any parabolic subgroup of ${\rm SL}_n(k)$ is conjugate to one of these.  The subgroup $L_{\mathbb n}$ of block diagonal matrices with diagonal block sizes $n_1,\ldots, n_r$ down the leading diagonal is a Levi subgroup of $P$.  The subgroup $U_{\mathbf n}$ of block upper unitriangular matrices with diagonal block sizes $n_1,\ldots, n_r$ down the leading diagonal is the unipotent radical of $P$.  Two extreme cases: if $r= n$ and $n_1=\cdots = n_r= 1$ then $P_{\mathbf n}$ is the Borel subgroup $B_n$, $L_{\mathbf n}$ is the maximal torus $D_n$ and $U_{\mathbf n}$ is $U_n$, while if $r= 1$ and $n_1= n$ then $P_{\mathbf n}= L_{\mathbf n}= {\rm SL}_n(k)$ and $U_{\mathbf n}= 1$.

A proper parabolic subgroup $P$ of $H$ is never reductive.  One can often, however, prove results by induction on ${\rm dim}(H)$, by passing from $H$ to a Levi subgroup $L$ of a proper parabolic subgroup $P$ of $H$.  In general, $H$ can have very complicated connected reductive subgroups, but Levi subgroups of $H$ are well-behaved: for instance, they always contain a maximal torus of $H$.

\medskip
\noindent {\bf Simple groups and semisimple groups:} An algebraic group $H$ is {\em simple} if it is infinite and connected and has no infinite proper normal subgroups.  In this case, any normal subgroup of $H$ is central, $Z(H)$ is finite and $H/Z(H)$ is simple as an abstract group.  (For instance, ${\rm SL}_n(k)$ is simple if $n\geq 2$---note that ${\rm SL}_1(k)$ is the trivial group!)  An algebraic group is {\em semisimple} if it is connected and reductive and has finite centre.  Up to isogeny, a semisimple group is a finite product of simple groups, and a connected reductive group is a finite product of simple groups with a central torus; the simple groups that appear in this factorisation are called the {\em simple components}.  If $H$ is connected and reductive then $Z(H)$ consists of semisimple elements and is a finite extension of a torus; moreover, the commutator subgroup $[H,H]$ is semisimple and has the same simple components as $H$.

\medskip
\noindent {\bf The structure theory of reductive groups}: A connected reductive group is completely determined by specifying the dimension of the torus $Z(H)^0$ and some combinatorial information called the {\em root datum}.  The root datum is completely determined (at least for semisimple groups) by a Dynkin diagram and some extra information which is closely analogous to the fundamental group of a Lie group.  A semisimple group is simple if and only if the corresponding Dynkin diagram is irreducible.  There are four infinite families of irreducible Dynkin diagrams, yielding the simple groups of classical type: type $A_n$ (which corresponds to ${\rm SL}_{n+1}(k)$ up to isogeny), $B_n$ for $n\geq 2$ (the special orthogonal group ${\rm SO}_{2n+1}(k)$\footnote{For $p> 2$; there are some subtleties over a field of characteristic 2.}), $C_n$ for $n\geq 3$ (the symplectic group ${\rm Sp}_{2n}(k)$) and $D_n$ for $n\geq 4$ (the special orthogonal group ${\rm SO}_{2n}(k)$\footnote{For $p> 2$; there are some subtleties over a field of characteristic 2.}).  There are also five so-called exceptional irreducible Dynkin diagrams: $G_2$, $F_4$, $E_6$, $E_7$ and $E_8$; the corresponding groups are said to be of exceptional type.

\medskip
\noindent {\bf Good and bad primes:} If ${\rm char}(k)= 0$ then algebraic groups are well-behaved in many important ways.  The general philosophy is that an algebraic group is well-behaved if $p:= {\rm char}(k)$ is ``large enough''.  In particular, if $H$ is connected and reductive then one can define the notion of a {\em good prime} using the combinatorics of the root system.  A prime is {\em bad} if it is not good.  Here is the list of bad primes in each case: 2 is bad for simple groups of all types except $A_n$; 3 is bad for types $G_2$, $F_4$, $E_6$, $E_7$ and $E_8$, and 5 is bad for type $E_8$.  All primes are good for type $A_n$, but we say $p$ is {\em very good} for type $X$ if either $X\neq A_n$ and $p$ is good for type $X$, or $X= A_n$ and $p$ does not divide $n+1$.  The prime $p$ is good (very good) for a reductive group $H$ if it is good (very good) for every simple component of $H$.
In our example above of $Z({\rm SL}_n(k))$ and ${\mathfrak z}(\mathfrak{sl}_n(k))$, ${\rm SL}_n(k)$ is well-behaved as long as ${\rm char}(k)= 0$ or ${\rm char}(k)$ is very good for ${\rm SL}_n(k)$.

\medskip
\noindent {\bf What I left out:} There is a huge gap in the above summary: we have not discussed roots and related topics (the Weyl group, root systems, the root datum).  If you want to learn more about this elegant and powerful structure theory, I urge you to read the books of Humphreys \cite{hum}, Borel \cite{borel} or Springer \cite{springer}.  For instance, the parabolic subgroups containing a fixed Borel subgroup have a combinatorial characterisation in terms of roots.

\section{$G$-complete reducibility}
\label{sec:Gr}

\subsection{Definition and examples}

We make the following assumption for convenience.

\begin{assumption}
 {\bf From now on, we assume $G$ is a connected reductive group unless otherwise stated.}
\end{assumption}

\noindent One can, however, extend the definition of $G$-complete reducibility to subgroups of non-connected reductive groups \cite[Sec.\ 6]{BMR}.  See Section~\ref{sec:nonconn} below for a very brief discussion.

\begin{defn}
\label{defn:Gcr}
 Let $H$ be a subgroup of $G$.  Then $H$ is {\em $G$-completely reducible} ($G$-cr) if for any parabolic subgroup $P$ of $G$ that contains $H$, there is a Levi subgroup $L$ of $P$ such that $L$ contains $H$.  We say that $H$ is {\em $G$-irreducible} ($G$-ir) if $H$ is not contained in any proper parabolic subgroup of $G$.  It is immediate that if $H$ is $G$-ir then $H$ is $G$-cr (why?).
\end{defn}

\noindent Here is a useful observation.  Since any two Levi subgroups of a parabolic subgroup $P$ are $R_u(P)$-conjugate, a subgroup $H$ of $G$ is $G$-cr if and only if the following holds: for any parabolic subgroup $P$ of $G$ that contains $H$ and for any Levi subgroup $L$ of $P$, $H$ is $R_u(P)$-conjugate to a subgroup of $L$.

\begin{ex}
\label{ex:DLn_SLn_crit}
 Here is a criterion for $G$-complete reducibility when $G= {\rm GL}_n(k)$ or ${\rm SL}_n(k)$:\smallskip\\
 $(*)$\ \ If $H\leq G$ then $H$ is $G$-cr if and only if the inclusion $i\colon H\ra G$ is a completely reducible representation.\smallskip\\
We see that our definition of $G$-complete reducibility coincides with the usual notion of complete reducibility from representation theory when $G= {\rm GL}_n(k)$ or ${\rm SL}_n(k)$.  The representation theory of algebraic groups now gives us lots of examples of $G$-cr subgroups and non-$G$-cr subgroups of general linear and special linear groups.  For instance, the adjoint representation $\rho$ of ${\rm SL}_2(k)$ on its Lie algebra $\mathfrak{sl}_2(k)$ is completely reducible---in fact, irreducible---if ${\rm char}(k)\neq 2$, while if ${\rm char}(k)= 2$ then $\rho$ is not completely reducible.
It follows that ${\rm Im}(\rho)$ is ${\rm GL}_3(k)$-ir if ${\rm char}(k)\neq 2$ and is not ${\rm GL}_3(k)$-cr if ${\rm char}(k)= 2$.

 Recall the description of parabolic subgroups and Levi subgroups of $G$ in terms of block upper triangular and block diagonal subgroups.  Below we prove $(*)$ building on the arguments of Example~\ref{ex:GLn_rep}.  There is one subtlety, though.  In Example~\ref{ex:GLn_rep}, we showed that if $H\leq {\rm GL}_n(k)$ is completely reducible in the sense of representation theory and $H\leq P$ then $H$ is ${\rm GL}_n(k)$-conjugate to a subgroup of $L$.  To show that $H$ is ${\rm GL}_n(k)$-cr, we need to prove that $H$ is $R_u(P)$-conjugate to a subgroup of $L$.  It is {\bf not} enough to know that $H$ is ${\rm GL}_n(k)$-conjugate to a subgroup of $L$.  For instance, take $n=3$ and $H$ to be the subgroup of matrices of the form $\left(
\begin{array}{ccc}
 1 & 0 & *\\
 0 & 1 & 0\\
 0 & 0 & 1\\
\end{array}
\right)$, which sits inside the parabolic subgroup $P$ of matrices of the form $\left(
\begin{array}{ccc}
 * & * & *\\
 * & * & *\\
 0 & 0 & *\\
\end{array}
\right)$.  Then $gHg^{-1}$ is contained in the Levi subgroup $L$ of matrices of the form $\left(
\begin{array}{ccc}
 * & * & 0\\
 * & * & 0\\
 0 & 0 & *\\
\end{array}
\right)$, where $g= \left(
\begin{array}{ccc}
 1 & 0 & 0\\
 0 & 0 & 1\\
 0 & 1 & 0\\
\end{array}
\right)$, but $H$ is not $R_u(P)$-conjugate to a subgroup of $L$ since $H\subseteq R_u(P)$. 

 In the setting of Example~\ref{ex:GLn_rep}, when $P$ is the stabiliser of a single subspace $W$, suppose $H$ stabilises a complementary subspace $W'$ to $W$.  We can choose a new basis to have the form $f_1,\ldots, f_m, f_{m+1}, \ldots, f_n$, where $f_{m+1},\ldots, f_n$ span $W'$, $f_i= e_i$ for $1\leq i\leq m$, and for $m+1\leq i\leq n$ each $f_i$ is of the form $e_i$ plus some linear combination of the $e_j$ for $1\leq j\leq m$.  The change of basis matrix is then block upper triangular with an $m\times m$ identity block followed by an $(n-m)\times (n-m)$ identity block down the leading diagonal, and such a matrix belongs to $R_u(P)$.  Hence $H$ is $R_u(P)$-conjugate to a subgroup of $L$.  Conversely, we can prove in a similar fashion that if $H$ is $R_u(P)$-conjugate to a subgroup of $L$ then $H$ stabilises a complementary subspace to $W$.  In general, $P$ will be the stabiliser of a flag of subspaces and a slightly more complicated argument is needed, but the idea is the same.  The equivalence in $(*)$ now follows.
\end{ex}

\begin{ex}
\label{ex:BT}
 Let $U$ be a nontrivial unipotent subgroup of $G$.  The following construction is due to Borel-Tits.  Define $N_1= N_G(U)$, $U_1= UR_u(N_1)$ and then define $N_m$ and $U_m$ inductively by $N_{m+1}= N_G(U_m)$ and $U_{m+1}= U_mR_u(N_{m+1})$.  Then $U\leq U_1\leq U_2\leq\cdots$, and one can show that this sequence eventually stabilises (exercise), so the sequence $N_1, N_2,\ldots$ eventually stabilises.  Let ${\mathcal P}(U)$ be the eventual stable value of the latter sequence.  It can be shown that ${\mathcal P}(U)$ is a parabolic subgroup of $G$ and $U\leq R_u({\mathcal P}(U))$ \cite[30.3]{hum} (this is not so easy).  We see that $U$ is not contained in any Levi subgroup of ${\mathcal P}(U)$, since every Levi subgroup of ${\mathcal P}(U)$ has trivial intersection with $R_u({\mathcal P}(U))$.  It follows that $U$ is not $G$-cr.  We say that the parabolic subgroup ${\mathcal P}(U)$ is a {\em witness} that $U$ is not $G$-cr.
 
 The parabolic subgroup ${\mathcal P}(U)$ is canonical in the following sense: if $\phi\in {\rm Aut}(G)$ and $\phi(U)= U$ then $\phi({\mathcal P}(U))= {\mathcal P}(U)$ (this is clear from the construction).  In particular, $N_G(U)$ normalises ${\mathcal P}(U)$, so $N_G(U)\leq {\mathcal P}(U)$.  We conclude that if $H\leq G$ and $H$ has a nontrivial normal subgroup $U$ then $H$ is not $G$-cr, because ${\mathcal P}(U)$ is a witness that $H$ is not $G$-cr.
 
 Corollary: a $G$-cr subgroup of $G$ must be reductive.
\end{ex}

\begin{ex}
\label{ex:linred}
 Recall that an algebraic group $H$ is linearly reductive if every representation of $H$ is completely reducible.  It can be shown using the cohomological ideas mentioned in Section~\ref{sec:history} that if $H$ is a linearly reductive subgroup of $G$ then $H$ is $G$-cr.  If $H$ is linearly reductive then it is reductive.  The converse is also true in characteristic 0: $H$ is reductive if and only if it is linearly reductive.  In particular, any finite group is linearly reductive in characteristic 0.  In characteristic $p> 0$, however, $H$ is linearly reductive if and only if $H^0$ is a torus and $H/H^0$ has order coprime to ${\rm char}(k)$ (equivalently: $H$ is linearly reductive if and only if every element of $H$ is semisimple); in particular, if $H$ is finite then $H$ is linearly reductive if and only if $|H|$ is coprime to $p$.
\end{ex}

Example~\ref{ex:linred} says that $G$-complete reducibility is less interesting in characteristic 0: a subgroup $H$ of $G$ is $G$-cr if and only if it is reductive.  For this reason, we make the following assumption to simplify the exposition:

\begin{assumption}
 {\bf From now on, we assume that $p:= {\rm char}(k)$ is positive.}
\end{assumption}

\noindent Almost everything said below, however, also holds in characteristic 0 with suitable modifications.

\subsection{Some history}
\label{sec:history}

The notion of $G$-complete reducibility was first introduced by Serre \cite{serre3} (cf.\ also \cite{serre1}).  He gave an interpretation of $G$-complete reducibility in terms of spherical buildings; we briefly recall this now.  We define the spherical building $X(G)$ of $G$ to be the simplicial complex whose simplices are the parabolic subgroups of $G$, ordered by reverse inclusion.  We identify $X(G)$ with its geometric realisation.  The group $G$ acts on $X(G)$ by conjugation.  Let $H\leq G$; then the fixed-point set $X(G)^H$ is a subcomplex of $X(G)$, consisting of all the parabolic subgroups of $G$ that contain $H$.  Serre showed that $H$ is $G$-cr if and only if $X(G)^H$ is not contractible in the usual sense of topology (see \cite[Thm.\ 2]{serre3}).  He also proved the following result using building-theoretic methods \cite[Prop.\ 3.2]{serre2}.

\begin{prop}
\label{prop:Gcr_Lcr}
 Let $L$ be a Levi subgroup of $G$ and let $H\leq L$.  Then $H$ is $G$-cr if and only if $H$ is $L$-cr.
\end{prop}

Liebeck and Seitz studied the subgroup structure of simple groups $G$ of exceptional type \cite{liebeckseitz0}.  They proved that if $p> 7$ then every connected reductive subgroup of $G$ is $G$-cr.  A key tool in their proof was nonabelian cohomology.  Let $P$ be a proper parabolic subgroup of $G$ (where $G$ is an arbitrary connected reductive group once again), with unipotent radical $V$, and let $L$ be a Levi subgroup of $P$.  Let $H\leq P$.  Then $H$ acts on $V$ via the rule $h\cdot v= c_L(h)v c_L(h)^{-1}$, and we can form the nonabelian 1-cohomology $H^1(H,V)$ of $H$ with coefficients in $V$.  If $H^1(H,V)$ vanishes then $H$ is $V$-conjugate to a subgroup of $L$.
By a result of Richardson, $H^1(H,V)$ is trivial if $H$ is linearly reductive, so we deduce that linearly reductive subgroups of $G$ are always $G$-cr.

The group $V$ has a filtration $V= V_0\supset V_1\supset\cdots \supset V_r= 1$ by $H$-stable normal subgroups such that each quotient $V_i/V_{i+1}$ is a vector space and the induced action of $H$ on $V_i/V_{i+1}$ is linear.  Liebeck and Seitz used the abelian cohomology theory of reductive groups to study the 1-cohomology $H^1(H, V_i/V_{i+1})$ of each layer $V_i/V_{i+1}$ and prove that $H^1(H,V)$ vanishes when $G$ is simple and of exceptional type, $H$ is connected and reductive and $p> 7$.

David Stewart and others have used cohomological techniques to study the subgroup structure of simple groups $G$ in all characteristics (see, e.g., \cite{stewartG2}, \cite{stewartF4}, \cite{stewartTAMS}).  If $p$ is small then $H^1(H,V)$ need not vanish, so $G$ can admit connected reductive subgroups that are not $G$-cr.

\subsection{Further results and constructions}

\begin{prop}
\label{prop:product}
 Let $G_1$ and $G_2$ be connected reductive groups, and let $H$ be a subgroup of $G_1\times G_2$.  Let $\pi_i\colon G_1\times G_2\ra G_i$ be the canonical projection.  Then $H$ is $(G_1\times G_2)$-cr if and only if $\pi_1(H)$ is $G_1$-cr and $\pi_2(H)$ is $G_2$-cr.
\end{prop}

\begin{proof}
 Standard structure theory for reductive groups implies that the parabolic subgroups of $G_1\times G_2$ are precisely the subgroups of the form $P_1\times P_2$, where $P_i$ is a parabolic subgroup of $G_i$ for each $i$, and the Levi subgroups of $P_1\times P_2$ are precisely the subgroups of the form $L_1\times L_2$, where $L_i$ is a Levi subgroup of $P_i$ for each $i$.  The result now follows easily (exercise).
\end{proof}

\begin{rem}
 By a similar argument, we can prove the following: if $\phi\colon G_1\ra G_2$ is an isogeny of connected reductive groups and $H\leq G_1$ then $H$ is $G_1$-cr if and only if $\phi(H)$ is $G_2$-cr.  For the structure theory implies that if $P\leq G_1$ then $P$ is a parabolic subgroup of $G_1$ if and only if $\phi(P)$ is a parabolic subgroup of $G_2$, and in this case $\phi^{-1}(\phi(P))= P$.  Likewise, if $L\leq G_1$ then $L$ is a parabolic subgroup of $G_1$ if and only if $\phi(L)$ is a parabolic subgroup of $G_2$, and in this case $\phi^{-1}(\phi(L))= L$.
\end{rem}

It is natural to ask the following question.

\begin{qn}
\label{qn:chain}
 Let $M$ be a connected reductive subgroup of $G$ and let $H\leq M$.  Is it true that $H$ is $M$-cr if and only if $H$ is $G$-cr?
\end{qn}

The answer is yes if $M$ is a Levi subgroup of $G$, by Proposition~\ref{prop:Gcr_Lcr}.  In general, the answer is no.  For example, fusion problems can arise: if $P$ is a parabolic subgroup of $M$ with Levi subgroup $L$ and $H\leq P$ then $H$ might be $G$-conjugate to a subgroup of $L$, but not $M$-conjugate to a subgroup of $L$.  For a simple counter-example to Question~\ref{qn:chain} in one direction, just take $M$ to be non-$G$-cr and $H$ to be $M$: then $H$ is not $G$-cr but $H$ is trivially $M$-cr.  Counter-examples in the other direction are harder to find.  The first example below is due to Bate-Martin-R\"ohrle-Tange \cite[Sec.\ 7]{BMRT}; the second to Liebeck \cite[Ex.\ 3.45]{BMR}.

\begin{ex}
\label{ex:BMRT_TAMS}
 Let $p= 2$ and let $G$ be a simple group of type $G_2$.  Let $M$ be the short root subgroup of type $A_1\times A_1$.  Then $M$ has a subgroup $H\iso S_3$ such that $H$ is $G$-cr but not $M$-cr.  (Uchiyama has found similar examples for $G$ simple of type $E_6$, $E_7$ and $E_8$ in characteristic 2 \cite{uchiyama2}, \cite{uchiyama3}.)
\end{ex}

\begin{ex}
 Suppose $p= 2$.  Let $m\geq 2$ be even.  The symplectic group $M:= {\rm Sp}_{2m}$ is a subgroup of $G:= {\rm SL}_{2m}$, and $K:= {\rm Sp}_m\times {\rm Sp}_m$ is a subgroup of ${\rm Sp}_{2m}$.  Let $H$ be ${\rm Sp}_m$ diagonally embedded in ${\rm Sp}_m\times {\rm Sp}_m$.  We have a chain of inclusions
 $$ H\leq K\leq M\leq G. $$
 It can be shown that $H$ is $K$-cr---this follows from Proposition~\ref{prop:product}---and $G$-cr, but not $M$-cr: so we have a counter-example to {\bf both} directions of Question~\ref{qn:chain}!
\end{ex}

We finish this section with an application to finite subgroups of $G$.  If $F$ is a finite group and $p$ does not divide $|F|$ then $G$ has only finitely many conjugacy classes of subgroups isomorphic to $F$, by Maschke's Theorem (see \cite{slodowy}).  This fails in general: for instance, ${\rm SL}_2(k)$ has infinitely many conjugacy classes of subgroups isomorphic to $C_p\times C_p$, where $C_p$ is the cyclic group of order $p$ (exercise).  But we have the following result \cite[Thm.\ 1.2]{martin}, \cite[Cor.\ 3.8]{BMR}, which is based on work of E.B.~Vinberg \cite{vinberg}.  The proof involves ideas from geometric invariant theory, and is difficult.

\begin{thm}
\label{thm:fin_cr_rep}
 Let $F$ be a finite group.  Then $G$ has only finitely many conjugacy classes of $G$-cr subgroups that are isomorphic to $F$.
\end{thm}

\section{Geometric invariant theory}
\label{sec:GIT}

Let $H$ be an algebraic group and let $X$ be an affine variety.  An {\em action} of $H$ on $X$ is a function $H\times X\ra X$ which is a left action of the abstract group $H$ on the set $X$ and is also a morphism of varieties.  We call $X$ an {\em $H$-variety}.  Given $x\in X$, we denote the stabiliser of $x$ by $H_x$ and the orbit of $x$ by $H\cdot x$.  Every stabiliser $H_x$ is a closed subgroup of $H$.  We define the orbit map $\kappa_x\colon H\ra H\cdot x$ by $\kappa_x(h)= h\cdot x$.  The closure $\ovl{H\cdot x}$ is a union of $H$-orbits and $H\cdot x$ is an open subset of $\ovl{H\cdot x}$; so $H\cdot x$ is a quasi-affine variety and it makes sense to speak of $\dim(H\cdot x)$ (note that every irreducible component of $H\cdot x$ has the same dimension since $H$ acts transitively on $H\cdot x$).  Every orbit contained in $\ovl{H\cdot x}\backslash H\cdot x$ has dimension less than that of $H\cdot x$.  It follows that orbits of minimal dimension are closed.  In particular, if every orbit has the same dimension then all the orbits are closed.  

An {\em $H$-module} is a finite-dimensional vector space $V$ on which $H$ acts linearly (so the action of $H$ comes from a rational representation $\rho\colon H\ra {\rm GL}(V)$).  It is convenient to be able to reduce from arbitrary $H$-varieties to the special case of $H$-modules.  It turns out that any $H$-variety $X$ can be embedded $H$-equivariantly inside an $H$-module.  To see this, observe that $H$ acts on the co-ordinate ring $k[X]$ by $k$-algebra automorphisms: $(h\cdot f)(x):= f(h^{-1}\cdot x)$.  In particular, this action is $k$-linear.  One can show that any finite subset of $k[X]$ is contained in a finite-dimensional $H$-stable subspace $W$ of $k[X]$.  Then $W$ is an $H$-module, the dual space $V:= W^*$ is also an $H$-module, and there is a canonical $H$-equivariant map from $X$ to $V$ (exercise).  If $W$ is large enough in an appropriate sense then this map is an embedding.

Geometric invariant theory (GIT) is the study of this set-up (see \cite[Ch.\ 3]{newstead} for a good introduction).  A fundamental question is the following: if $X$ is an $H$-variety, does there exist a ``reasonable'' quotient variety?  We can answer this question when the group concerned is reductive.

\begin{thm}
 Let $X$ be a $G$-variety.  There exist an affine variety $X/\!\!/G$ and a $G$-invariant morphism $\pi_X\colon X\ra X/\!\!/G$ which satisfies the following universal mapping property: if $\psi\colon X\ra Y$ is a $G$-invariant morphism of varieties then there is a unique morphism $\psi_G\colon X/\!\!/G\ra Y$ such that $\psi= \psi_G\circ \pi_X$. 
\end{thm}

\noindent The quotient variety $X/\!\!/G$ is---by definition---the affine variety whose co-ordinate ring is $k[X]^G$, the ring of invariants for the $G$-action on $k[X]$; it is a deep theorem that $k[X]^G$ is a finitely generated $k$-algebra.  The map $\pi_X$ comes from the inclusion of $k[X]^G$ in $k[X]$.

\begin{ex}
\label{ex:squared}
 All of the above holds for nonconnected reductive groups as well.  We temporarily suspend our assumption that $G$ is connected and look at a simple example where $G$ is finite.  Suppose $p\neq 2$.   Let $X= k$ and let $G= C_2= \langle g\mid g^2= 1\rangle$, acting on $X$ by $g\cdot x= -x$.  Now $k[X]$ is the polynomial ring $k[T]$, and $G$ acts on $k[T]$ by $g\cdot T= -T$ (and $g\cdot b= b$ for $b\in k$).  It is clear that $k[T]^G= k[T^2]$, which is also a polynomial ring in one indeterminate, so $X/\!\!/G\iso k$.  The map $\pi_G\colon X\ra X/\!\!/G$ comes from the inclusion of $k[T^2]= k[X]^G$ in $k[T]= k[X]$.  So $\pi_G\colon k\ra k$ is given by $\pi_G(b)= b^2$.
\end{ex}

In Example~\ref{ex:squared}, $X/\!\!/G$ is a set-theoretic quotient of $X$.  Unfortunately, this is false in general: for if $z\in X/\!\!/G$ and $\pi_X^{-1}(z)$ consists of a single orbit $G\cdot x$ then $G\cdot x$ must be closed, since $\{z\}$ is closed and morphisms of varieties are continuous.  Hence if $G\cdot x$ is not closed then $G\cdot x$ cannot be a fibre of $\pi_X$.  This is not an issue in Example~\ref{ex:squared} as all the orbits there are closed, but the next example illustrates what can go wrong.

\begin{ex}
 Let $X= k^n$ and let $G= k^*$ acting on $X$ by scalar multiplication in the usual way.  Then the only closed orbit is $\{0\}$, so $X/\!\!/G$ consists of just a single point.
\end{ex}

It is crucial, therefore, to know which orbits are closed.  The celebrated Hilbert-Mumford Theorem gives a criterion for this in terms of cocharacters. 

\begin{defn}
 A {\em cocharacter} (or {\em 1-parameter subgroup}) of $G$ is a homomorphism $\lambda\colon k^*\ra G$.
\end{defn}

\noindent We denote by $Y(G)$ the set of cocharacters.  There is a close analogy with $X(G)$, but note that pointwise multiplication does {\bf not} in general give a well-defined binary operation on $Y(G)$.  If $T$ is a maximal torus of $G$, however, then $Y(T)$---the set of cocharacters whose image lies in $T$---is a free abelian group under pointwise multiplication; if $T\iso (k^*)^m$ then $Y(T)$ has rank $m$.  We use additive notation for $Y(T)$: in particular, if $\lambda\in Y(T)$ and $n\in {\mathbb Z}$ then $n\lambda$ denotes the cocharacter given by $(n\lambda)(a)= \lambda(a)^n$.  If $\lambda\in Y(G)$ then ${\rm Im}(\lambda)$ is a torus, so $\lambda\in Y(T)$ for some maximal torus $T$ of $G$.  There is an action of $G$ on $Y(G)$ given by $(g\cdot \lambda)(a)= g\lambda(a)g^{-1}$; if $\lambda$ belongs to $Y(T)$ then $g\cdot \lambda$ belongs to $Y(gTg^{-1})$.

There is a natural nondegenerate bilinear pairing between $Y(T)$ and $X(T)$, defined as follows.  If $\lambda\in Y(T)$ and $\chi\in X(T)$ then $\chi\circ \lambda$ is an endomorphism of $k^*$, so is of the form $a\mapsto a^n$ for some $n\in {\mathbb Z}$; we set $\langle \lambda, \chi\rangle= n$.  Given $g\in G$ and $\chi\in X(T)$, we define $g\cdot \chi\in X(gTg^{-1})$ by $(g\cdot \chi)(gtg^{-1})= \chi(t)$ for $t\in T$.  A straightforward calculation shows that
\begin{equation}
\label{eqn:pairing_eqvrce}
 \langle g\cdot \lambda, g\cdot \chi\rangle= \langle \lambda, \chi\rangle
\end{equation}
for every $\lambda\in Y(T)$, $\chi\in X(T)$ and $g\in G$.  It is also easy to check that if $V$ is a $G$-module and $\chi$ is a weight for $V$ with respect to $T$ then $g\cdot \chi$ is a weight of $V$ with respect to $gTg^{-1}$ and $V_{g\cdot \chi}= g\cdot V_\chi$.

We now introduce the notion of limits.

\begin{defn}
 Let $f\colon k^*\ra X$ be a morphism of varieties.  We say that {\em $\lim_{a\to 0} f(a)$ exists} if $f$ extends to a morphism $\widehat{f}\colon k\ra X$.  If the limit exists, we set $\lim_{a\to 0} f(a)= \widehat{f}(0)$.  Note that $\widehat{f}$, if it exists, is unique, because $k^*$ is dense in $k$.
\end{defn}

\begin{rem}
\label{rem:limit_props}
 The following results follow easily from the definition of limit and  the universal mapping property for products.\smallskip\\
 (a) If $f\colon k^*\ra X$ and $h\colon X\ra Y$ are morphisms of varieties and $x:= \lim_{a\to 0} f(a)$ exists then $\lim_{a\to 0} (h\circ f)(a)$ exists, and $\lim_{a\to 0} (h\circ f)(a)= h(x)$.\smallskip\\
 (b) If $f_1\colon k^*\ra X_1$ and $f_2\colon k^*\ra X_2$ are morphisms then $\lim_{a\to 0} (f_1\times f_2)(a)$ exists if and only if $x_1:= \lim_{a\to 0} f_1(a)$ and $x_2:= \lim_{a\to 0} f_2(a)$ exists, and in this case $\lim_{a\to 0} (f_1\times f_2)(a)= (x_1, x_2)$.
\end{rem}

\begin{ex}
\label{ex:basic_lim}
 Let $n\in {\mathbb Z}$.  Define $f\colon k^*\ra k$ by $f(a)= a^n$.  If $n> 0$ then the morphism $\widehat{f}\colon k\ra k$ given by $f(a)= a^n$ is an extension of $f$, so $\lim_{a\to 0} f(a)$ exists and equals $0^n= 0$.  Likewise, if $n= 0$ then $\lim_{a\to 0} f(a)$ exists and equals $1$ (as usual, we interpret $a^0$ as 1 for any $a$).  On the other hand, if $n< 0$ then $\lim_{a\to 0} f(a)$ does {\bf not} exist.  For suppose otherwise.  Define $h\colon k^*\ra k$ by $h(a)= a^{-n}$.  Then $fh$ is the constant function 1, so $\lim_{a\to 0} (fh)(a)= 1$.  By Remark~\ref{rem:limit_props}(b), $\lim_{a\to 0} (f\times h)(a)= \left(\lim_{a\to 0} f(a), \lim_{a\to 0} h(a)\right)$.  Since multiplication is a morphism from $k^2$ to $k$, Remark~\ref{rem:limit_props}(a) implies that $\lim_{a\to 0} (fh)(a)= \left(\lim_{a\to 0} f(a)\right) \left(\lim_{a\to 0} h(a)\right)$.  But $ \lim_{a\to 0} h(a)= 0$ as $-n> 0$, so we get a contradiction.
\end{ex}

\noindent Here is our main application.  Let $X$ be a $G$-variety, let $x\in X$ and let $\lambda\in Y(G)$.  We want to know when $\lim_{a\to 0} \lambda(a)\cdot x$ exists (here we take $f(a)= \lambda(a)\cdot x$).  Note that $\lim_{a\to 0} \lambda(a)\cdot x$, if it exists, belongs to $\ovl{G\cdot x}$ (easy exercise).

\begin{ex}
\label{ex:GL2_lim}
 Let $X= {\rm GL}_2(k)$ and let $G= {\rm GL}_2(k)$ acting by conjugation on $X$ (so: $g\cdot x= gxg^{-1}$).  Define $\lambda\in Y(G)$ by $\lambda(a)=
\left(
\begin{array}{ccc}
 a & 0\\
 0 & a^{-1}\\
\end{array}
\right)
$.  Let $x_1=
\left(
\begin{array}{ccc}
 1 & 1\\
 0 & 1\\
\end{array}
\right).
$  Then
$$ \lambda(a)\cdot x_1= \lambda(a)x_1\lambda(a)^{-1}=
\left(
\begin{array}{ccc}
 a & 0\\
 0 & a^{-1}\\
\end{array}
\right)
\left(
\begin{array}{ccc}
 1 & 1\\
 0 & 1\\
\end{array}
\right)
\left(
\begin{array}{ccc}
 a^{-1} & 0\\
 0 & a\\
\end{array}
\right)
=
\left(
\begin{array}{ccc}
 1 & a^2\\
 0 & 1\\
\end{array}
\right),
$$ so $\lim_{a\to 0} \lambda(a)\cdot x_1=
\left(
\begin{array}{ccc}
 1 & 0\\
 0 & 1\\
\end{array}
\right)
= I$ as $\lim_{a\to 0} a^2= 0$.  It follows that $\pi_X(x_1)= \pi_X(I)$: so $G\cdot x_1$ ``disappears'' (or, rather, coalesces with the closed orbit $G\cdot I$) in the quotient variety $X/\!\!/G$.  We see that $X/\!\!/G$ is not a set-theoretic quotient; this illustrates the problem discussed above.

Now let $x_2=
\left(
\begin{array}{ccc}
 1 & 0\\
 1 & 1\\
\end{array}
\right).
$  Then
$$ \lambda(a)\cdot x_2= \lambda(a)x_2\lambda(a)^{-1}=
\left(
\begin{array}{ccc}
 a & 0\\
 0 & a^{-1}\\
\end{array}
\right)
\left(
\begin{array}{ccc}
 1 & 0\\
 1 & 1\\
\end{array}
\right)
\left(
\begin{array}{ccc}
 a^{-1} & 0\\
 0 & a\\
\end{array}
\right)
=
\left(
\begin{array}{ccc}
 1 & 0\\
 a^{-2} & 1\\
\end{array}
\right),
$$ so $\lim_{a\to 0} \lambda(a)\cdot x_2$ does not exist  as $\lim_{a\to 0} a^{-2}$ does not exist.  In fact, although the above calculation does not show this, we have $\pi_X(x_2)= \pi_X(I)$ (why?).

More generally, if $x= \left(
\begin{array}{ccc}
 b & c\\
 d & e\\
\end{array}
\right)
$ then $\lambda(a)\cdot x=
\left(
\begin{array}{ccc}
 b & a^2c\\
 a^{-2}d & e\\
\end{array}
\right)
$, so $\lim_{a\to 0} \lambda(a)\cdot x$ exists if and only if $x\in B_2$, and $\lim_{a\to 0} \lambda(a)\cdot x= x$ if and only if $x\in D_2$.
\end{ex}

\begin{ex}
\label{ex:GLn_lim}
 Let $X= {\rm GL}_n(k)$ and let $G= {\rm GL}_n(k)$ acting by conjugation on $X$.  Define $\lambda\in Y(G)$ by $\lambda(a)= {\rm diag}(a^{r_1},\ldots, a^{r_1}, a^{r_2},\ldots, a^{r_2},\ldots, a^{r_t},\ldots, a^{r_t})$.  Here ``${\rm diag}$'' denotes the diagonal matrix with the specified entries, $t$ is a positive integer, the $r_i$ are integers satisfying $r_1> r_2> \cdots > r_t$, and each term $a^{r_i}$ appears $m_i$ times, where $(m_1,\ldots, m_t):= {\mathbf m}$ is a $t$-tuple of positive integers such that $m_1+\cdots + m_t= n$.  A calculation like the one in Example~\ref{ex:GL2_lim} shows that if $x\in {\rm GL}_n(k)$ then $\lim_{a\to 0} \lambda(a)\cdot x$ exists if and only if $x\in P_{\mathbf m}$, and $\lim_{a\to 0} \lambda(a)\cdot x= x$ if and only if $x\in L_{\mathbf m}$.
\end{ex}

\begin{ex}
 We now give the characterisation of parabolic subgroups and their Levi subgroups promised earlier.  Let $\lambda\in Y(G)$.  Set
 $$ P_\lambda= \left\{g\in G\ \left| \lim_{a\to 0} \lambda(a)g\lambda(a)^{-1}\ \mbox{exists}\right\}\right. $$
 and set $L_\lambda= C_G({\rm Im}(\lambda))$.  It can be shown that $P_\lambda$ is a parabolic subgroup of $G$ and $L_\lambda$ is a Levi subgroup of $P_\lambda$.  Moreover, if $P$ is a parabolic subgroup of $G$ and $L$ is a Levi subgroup of $P$ then there exists $\lambda\in Y(G)$ such that $P= P_\lambda$ and $L= L_\lambda$.  For any $n\in {\mathbb N}$, $P_{n\lambda}= P_\lambda$ and $L_{n\lambda}= L_\lambda$.  Define $c_\lambda\colon P_\lambda\ra G$ by $c_\lambda(g)= \lim_{a\to 0} \lambda(a)g\lambda(a)^{-1}$; then $c_\lambda(P_\lambda)\leq L_\lambda$ and $c_\lambda(g)= c_{L_\lambda}(g)$ for all $g\in P_\lambda$.  In particular, $L_\lambda= \{g\in P_\lambda\mid c_\lambda(g)= g\}$ and $R_u(P_\lambda)= \{g\in P_\lambda\mid c_\lambda(g)= 1\}$.  We have $P_\lambda= G$ if and only if $L_\lambda= G$ if and only if ${\rm Im}(\lambda)\subseteq Z(G)^0$.
 
 Suppose $M$ is a connected reductive subgroup of $G$ and $\lambda\in Y(M)$.  We write $P_\lambda(M)$ for the set $\{m\in M\mid \lim_{a\to 0} \lambda(a)m\lambda(a)^{-1}\ \mbox{exists}\}$: that is, $P_\lambda(M)$ is the parabolic subgroup constructed above, but for $M$.  Likewise we write $L_\lambda(M)$ for $C_M({\rm Im}(\lambda))$.  It follows from the previous paragraph that $P_\lambda(M)= P_\lambda\cap M$, $L_\lambda(M)= L_\lambda\cap M$ and $R_u(P_\lambda(M))= R_u(P_\lambda)\cap M$.  For brevity, we write just $P_\lambda$ and $L_\lambda$ instead of $P_\lambda(G)$ and $L_\lambda(G)$.
\end{ex}

The next result follows easily from Remark~\ref{rem:limit_props}(a).

\begin{lem}
\label{lem:lim_P_dot_x}
 Let $X$ be a $G$-variety, let $x\in X$ and let $\lambda\in Y(G)$ such that $x':= \lim_{a\to 0} \lambda(a)\cdot x$ exists.  Then for any $g\in P_\lambda$, $\lim_{a\to 0} \lambda(a)\cdot (g\cdot x)$ exists and equals $c_\lambda(g)\cdot x'$.
\end{lem}

\begin{ex}
\label{ex:module_lim}
 It's easier to see what's going on with limits when $X$ is a $G$-module.  Let $V$ be a $G$-module and let $\lambda\in Y(G)$.  Choose a maximal torus $T$ of $G$ such that $\lambda\in Y(T)$.  If $\chi\in \Phi_T(V)$ and $0\neq v\in V_\chi$ then
 $$ \lambda(a)\cdot v= \chi(\lambda(a))v= a^nv, $$
 where $n:= \langle \lambda, \chi\rangle$.  Hence $\lim_{a\to 0} \lambda(a)\cdot v$ exists if and only if $n\geq 0$.  If $n> 0$ then the limit is 0, while if $n= 0$ then the limit is $v$.
 
 Now let $v$ be an arbitrary element of $V$.  We have $v= \sum_{\chi\in {\rm supp}_T(v)} v_\chi$.  It follows from the discussion above that $\lim_{a\to 0} \lambda(a)\cdot v$ 
exists if and only if $\langle \lambda, \chi\rangle\geq 0$ for all $\chi\in {\rm supp}_T(v)$, and in this case $v':= \lim_{a\to 0} \lambda(a)\cdot v$ is given by $v'= \sum_{\chi\in F} v_\chi$, where $F= \{\chi\in {\rm supp}_T(v)\mid \langle \lambda, \chi\rangle= 0\}$.

Here is a useful consequence.  If $v':= \lim_{a\to 0} \lambda(a)\cdot v$ exists then $v'$ is fixed by ${\rm Im}(\lambda)$.  In particular, $\lim_{a\to 0} \lambda(a)\cdot v= v$ if and only if ${\rm Im}(\lambda)$ fixes $v$.  The same is true for points in an arbitrary $G$-variety $X$, since we can embed $X$ $G$-equivariantly in a $G$-module.
\end{ex}

We can now state the Hilbert-Mumford Theorem.  It says that in order to check whether a $G$-orbit is closed, it is sufficient to check whether the orbits $S\cdot x$ are closed for all 1-dimensional subtori $S$ of $G$.

\begin{thm}
\label{thm:HMT}
 Let $X$ be a $G$-variety and let $x\in X$.  Then there exists $\lambda\in Y(G)$ such that $x':= \lim_{a\to 0} \lambda(a)\cdot x$ exists and $G\cdot x'$ is closed.
\end{thm}

\begin{rem}
\label{rem:HMT_discussion}
 If $G\cdot x$ is already closed then we can just take $\lambda= 0$ and $x'= x$.  Otherwise, $\lambda$ is nontrivial and $x'$ does not lie in $G\cdot x$.  It follows that $S\cdot x$ is not closed and $x'\in \ovl{S\cdot x}$, where $S$ is the torus ${\rm Im}(\lambda)$.
 
 We can extract from the Hilbert-Mumford Theorem a useful criterion for an orbit to be closed: $G\cdot x$ is closed if and only if for all $\lambda\in Y(G)$ such that $x':= \lim_{a\to 0} \lambda(a)\cdot x$ exists, $x'$ belongs to $G\cdot x$.
\end{rem}

What can we say if $\lim_{a\to 0} \lambda(a)\cdot x$ exists but still lies in $G\cdot x$?  The following preliminary result is standard.

\begin{lem}
\label{lem:wts_increase}
 Let $V$, $\lambda$ and $T$ be as in Example~\ref{ex:module_lim}.  Let $\chi\in \Phi_T(V)$, let $v\in V_\chi$ and let $g\in P_\lambda$ (resp., $g\in R_u(P_\lambda)$).  Then for all $\chi'\in {\rm supp}_T(g\cdot v- v)$, $\langle \lambda, \chi'\rangle \geq \langle \lambda, \chi\rangle$ (resp., $\langle \lambda, \chi'\rangle> \langle \lambda, \chi\rangle$).
\end{lem}

\begin{thm}
\label{thm:mainconjalgclsd}
 Let $X$ be a $G$-variety and let $x\in X$.  Let $\lambda\in Y(G)$ and suppose $x':= \lim_{a\to 0} \lambda(a)\cdot x$ exists and $x'\in G\cdot x$.  Then $x'\in R_u(P_\lambda)\cdot x$.
\end{thm}

\begin{proof}
 (Sketch): We can embed $X$ $G$-equivariantly inside a $G$-module, so without loss we can assume $X$ is a $G$-module.  Choose a maximal torus $T$ such that $\lambda\in Y(T)$.  First we show that $x'\in P_\lambda\cdot x$.  The set $P_\lambda R_u(P_{-\lambda})$ contains the so-called ``big cell'', which is an open neighbourhood of 1 in $G$.  The orbit map $\kappa_{x'}\colon G\ra G\cdot x'$ is an open map, so $P_\lambda R_u(P_{-\lambda})\cdot x'$ contains an open neighbourhood of $x'$ in $G\cdot x'$.  Since $x'= \lim_{a\to 0} \lambda(a)\cdot x$ belongs to the closure of ${\rm Im}(\lambda)\cdot x$, there exists $a\in k^*$ such that $\lambda(a)\cdot x\in P_\lambda R_u(P_{-\lambda})\cdot x'$: say, $\lambda(a)\cdot x= gu\cdot x'$ for some $g\in P_\lambda$ and some $u\in R_u(P_{-\lambda})$.  This gives
 \begin{equation}
 \label{eqn:mainconjalgclsd}
  h\cdot x= u\cdot x',
 \end{equation}
where $h:= g^{-1}\lambda(a)\in P_\lambda$.  We have $\lim_{a\to 0} \lambda(a)\cdot (u\cdot x')= \lim_{a\to 0} \lambda(a)\cdot (h\cdot x)= c_\lambda(h)\cdot x'$ by Lemma~\ref{lem:lim_P_dot_x}: so $\lim_{a\to 0} \lambda(a)\cdot (u\cdot x')$ exists even though $\lim_{a\to 0} \lambda(a)u\lambda(a)^{-1}$ does not.
 
 Now $\lambda$ and $-\lambda$ fix $x'$, so by Lemma~\ref{lem:wts_increase} (applied to $-\lambda$), ${\rm supp}_T(u\cdot x'- x')$ consists of weights $\chi$ satisfying $\langle -\lambda, \chi\rangle> 0$.  Hence $\langle \lambda, \chi\rangle= -\langle -\lambda, \chi\rangle< 0$ for all $\chi\in {\rm supp}_T(u\cdot x'- x')$.  But $\lim_{a\to 0} \lambda(a)\cdot (u\cdot x')$ exists, so $\lim_{a\to 0} \lambda(a)\cdot (x'- u\cdot x')$ exists.  This forces ${\rm supp}_T(u\cdot x'- x')$ to be empty, which means that $u\cdot x'- x'= 0$.  So $u\cdot x'= x'$ and $x'= h\cdot x\in P_\lambda\cdot x$, as required.
 
 To finish, we prove that $x'\in R_u(P_\lambda)\cdot x$.  Write $h= lv$ with $l\in L_\lambda$ and $v\in R_u(P_\lambda)$.  Then $x'= lv\cdot x$, so $l^{-1}\cdot x'= v\cdot x$.  Taking $\lim_{a\to 0}$ of both sides and applying Lemma~\ref{lem:lim_P_dot_x} gives $l^{-1}\cdot x'= x'$.  So $x'= v\cdot x$, and we are done.
\end{proof}

\noindent {\bf Open Problem:} All of the above geometric invariant theory makes sense for a $G$-variety $X$ defined over an arbitrary field.  Does Theorem~\ref{thm:mainconjalgclsd} hold over an arbitrary field?  The proof above, which is taken from \cite[Thm.\ 3.3]{BMRT_git}, also works if the ground field is perfect.

\begin{cor}
\label{cor:still_in_orbit}
 Let $X$ be a $G$-variety, let $x\in X$ and let $\lambda\in Y(G)$.  Suppose $x':= \lim_{a\to 0} \lambda(a)\cdot x$ exists.  Then $x'$ belongs to $G\cdot x$ if and only if there exists $u\in R_u(P_\lambda)$ such that $\lambda$ fixes $u\cdot x$.  In this case, $x'= u\cdot x$.
\end{cor}

\begin{proof}
 Suppose $x'\in G\cdot x$.  Then $x'= u\cdot x$ for some $u\in R_u(P_\lambda)$, by Theorem~\ref{thm:mainconjalgclsd}, and $\lambda$ fixes $u\cdot x= x'$.  Conversely, suppose $u\in R_u(P_\lambda)$ and $\lambda$ fixes $u\cdot x$.  Taking the limit, we get $\lim_{a\to 0} \lambda(a)\cdot (u\cdot x)= x'$ by our usual argument.  But $\lambda$ fixes $u\cdot x$, so this limit is equal to $u\cdot x$.  Hence $x'= u\cdot x$, and we are done.
\end{proof}

\section{A geometric criterion for $G$-complete reducibility}
\label{sec:Gcr_GIT}

Let $m\in {\mathbb N}$.  The group $G$ acts on the variety $G^m$ by {\em simultaneous conjugation}:
$$ g\cdot (g_1,\ldots, g_m)= (gg_1g^{-1}, \ldots, gg_mg^{-1}) $$
for $g\in G$ and $(g_1,\ldots, g_m)\in G^m$.  We call the orbits of this action {\em conjugacy classes}.  Given ${\mathbf h}= (h_1,\ldots, h_m)\in G^m$, we define ${\mathcal G}({\mathbf h})$ to be the closure of the abstract group generated by $h_1,\ldots, h_m$.  We say that a subgroup $H$ of $G$ is {\em topologically finitely generated} if $H= {\mathcal G}({\mathbf h})$ for some $m\in {\mathbb N}$ and some ${\mathbf h}\in G^m$.

If $H= {\mathcal G}({\mathbf h})$ then $gHg^{-1}= {\mathcal G}(g\cdot {\mathbf h})$.  Richardson's fundamental insight is that one can study subgroups of $G$ up to $G$-conjugacy by studying conjugacy classes of tuples from $G^m$.  Since $G^m$ is a $G$-variety, this allows us to apply ideas from geometric invariant theory.  In particular, we can give a geometric criterion for subgroups of $G$ to be $G$-completely reducible.

\begin{thm}
\label{thm:Gcr_crit}
 Let ${\mathbf h}\in G^m$ and let $H= {\mathcal G}({\mathbf h})$.  Then $H$ is $G$-cr if and only if the conjugacy class $G\cdot {\mathbf h}$ is closed.
\end{thm}

\begin{proof}
 Write ${\mathbf h}= (h_1,\ldots, h_m)$.  Suppose $H$ is $G$-cr.  We prove that $G\cdot {\mathbf h}$ is closed.  Let $\lambda\in Y(G)$ such that ${\mathbf h}':= \lim_{a\to 0} \lambda(a)\cdot {\mathbf h}$ exists.  It is enough by the Hilbert-Mumford Theorem to show that ${\mathbf h}'\in G\cdot {\mathbf h}$.  Now $\lim_{a\to 0} (\lambda(a)h_1\lambda(a)^{-1},\ldots, \lambda(a)h_m\lambda(a)^{-1})$ exists, so $\lim_{a\to 0} \lambda(a)h_i\lambda(a)^{-1}$ exists for each $1\leq i\leq  m$ by Remark~\ref{rem:limit_props}(b).  Hence $h_i\in P_\lambda$ for each $1\leq i\leq m$, which means that $H\leq P_\lambda$.  As $H$ is $G$-cr, there exists $u\in R_u(P_\lambda)$ such that $uHu^{-1}\leq L_\lambda$.  This implies that $uh_iu^{-1}\in L_\lambda= C_G({\rm Im}(\lambda))$ for each $i$, so $\lambda$ fixes $u\cdot {\mathbf h}$, so ${\mathbf h}'= u\cdot {\mathbf h}\in G\cdot {\mathbf h}$, by Corollary~\ref{cor:still_in_orbit}.
 
 Conversely, suppose $G\cdot {\mathbf h}$ is closed.  Let $P$ be a parabolic subgroup of $H$ such that $H\leq P$.  Then $P= P_\lambda$ for some $\lambda\in Y(G)$, so $\lim_{a\to 0} \lambda(a)h_i\lambda(a)^{-1}$ exists for each $1\leq i\leq m$.  This implies that ${\mathbf h}':= \lim_{a\to 0} \lambda(a)\cdot {\mathbf h}$ exists and belongs to $L_\lambda^m$ by Remark~\ref{rem:limit_props}(b).  Now $G\cdot {\mathbf h}$ is closed, so ${\mathbf h}'$ belongs to $G\cdot {\mathbf h}$.  By Theorem~\ref{thm:mainconjalgclsd}, ${\mathbf h}'= u\cdot {\mathbf h}$ for some $u\in R_u(P_\lambda)$.  So ${\mathbf h}= u^{-1}\cdot {\mathbf h}'\in (u^{-1}L_\lambda u)^m$.  But this means that $H\leq u^{-1}L_\lambda u$, a Levi subgroup of $P_\lambda$.  We conclude that $H$ is $G$-cr, as required.
\end{proof}

\begin{rem}
 A historical note: Richardson defined the notion of a {\em strongly reductive} subgroup of $G$ \cite[Defn.\ 16.1]{rich4}, and proved that $H= {\mathcal G}({\mathbf h})$ is strongly reductive if and only if $G\cdot {\mathbf h}$ is closed \cite[Thm.\ 16.4]{rich4}.  Bate-Martin-R\"ohrle showed that a subgroup $H$ of $G$ is strongly reductive if and only if it is $G$-cr \cite[Thm.\ 3.1]{BMR}; their proof involved manipulations of parabolic and Levi subgroups, and did not go via the more general result Theorem~\ref{thm:mainconjalgclsd}.
\end{rem}

\begin{rem}
\label{rem:Ru_conj}
 The argument shows that if $H= {\mathcal G}({\mathbf h})$ is contained in  $P_\lambda$ and $H$ is $G$-conjugate to the subgroup $c_\lambda(H)$ of $L_\lambda$ then $H$ is $R_u(P)$-conjugate to $c_\lambda(H)$: cf.\ the discussion in Example~\ref{ex:DLn_SLn_crit}.
\end{rem}

\begin{rem}
\label{rem:topfg}
 We cannot {\em a priori} apply Theorem~\ref{thm:Gcr_crit} to an arbitrary subgroup of $G$, since not every subgroup of $G$ is topologically finitely generated.  For instance, if $k$ is the algebraic closure of ${\mathbb F}_p$ then any topologically finitely generated subgroup $H= {\mathcal G}((h_1,\ldots, h_m))$ is finite.  To see this, choose an embedding $i\colon G\ra {\rm GL}_n(k)$; then $i(h_1),\ldots, i(h_m)$ belong to ${\rm GL}_n({\mathbb F}_q)$ for some sufficiently large power $q$ of $p$, so they generate a finite subgroup.  On the other hand, if $k$ is transcendental over ${\mathbb F}_p$ then any reductive subgroup of $G$ is topologically finitely generated \cite[Lem.\ 9.2]{martin}.

 In practice, however, this annoying technicality does not cause us any serious problems.  One can show that any subgroup $H$ of $G$ contains a topologically finitely generated subgroup $H'$ such that $H'$ is contained in exactly the same parabolic subgroups and Levi subgroups as $H$ is; so as far as $G$-complete reducibility is concerned, $H$ and $H'$ behave in the same way (\cite[Lem.\ 2.10]{BMR}; see also \cite[Defn.\ 5.4]{BMRT_git}).  For other ways to deal with this problem, see \cite[Rem.\ 2.9]{BMR} and the discussion of uniform $S$-instability at the start of Section~\ref{sec:opt} below.  We will gloss over this subtlety and assume below that all the subgroups of $G$ we deal with are topologically finitely generated.
\end{rem}

Now and in the next section we will derive some consequences of Theorem~\ref{thm:Gcr_crit}.  Here is our first: if $H= {\mathcal G}({\mathbf h})\leq G$ is $G$-cr then $C_G(H)$ is reductive.  For $G\cdot {\mathbf h}$ is closed by Theorem~\ref{thm:Gcr_crit}, so the stabiliser $G_{\mathbf h}$ is reductive by a standard result from GIT \cite[Thm.\ A]{rich}; but it is clear that $G_{\mathbf h}= C_G(H)$.  A slight refinement of this argument \cite[Prop.\ 3.12]{BMR} shows that $N_G(H)$ is also reductive.  Later we show that $C_G(H)$ and $N_G(H)$ are actually $G$-cr.

To state our next results, we need the notions of a {\em separable subgroup} of $G$ and a {\em reductive pair}.  If $H\leq G$ then we denote by ${\mathfrak c}_{\mathfrak g}(H)$ the centraliser of $H$ in ${\mathfrak g}$ (that is, the fixed point set of $H$ in ${\mathfrak g}$ with respect to the adjoint action).  It is immediate that ${\rm Lie}(C_G(H))\subseteq {\mathfrak c}_{\mathfrak g}(H)$.  If equality holds, we say that $H$ is {\em separable} in $G$.  The reason for the terminology is that if $H= {\mathcal G}({\mathbf h})$ then $H$ is separable if and only if the orbit map $\kappa_{\mathbf h}\colon G\ra G\cdot {\mathbf h}$ is separable.  (Here is another equivalent condition: $H$ is separable if and only if the scheme-theoretic centraliser of $H$ in $G$ is smooth.)  By a result of Herpel \cite[Thm.\ 1.1]{herpel}, if $p$ is large enough---$p$ very good for $G$ will do---then every subgroup of $G$ is separable.  Any subgroup $H$ of ${\rm GL}_n(k)$ is separable.  To see this, recall that we can identify $\mathfrak{gl}_n(k)$ with $M_n(k)$, so the centraliser of $H$ in $\mathfrak{gl}_n(k)$ is the subalgebra $C:= \{A\in M_n(k)\mid Ah= hA\ \mbox{for all $h\in H$}\}$.  We have $C_{{\rm GL}_n(k)}(H)= C\cap {\rm GL}_n(k)$, so $C_{{\rm GL}_n(k)}(H)$ is an open subset of $C$ and therefore has the same dimension as $C$.

Let $M$ be a connected reductive subgroup of $G$.  We call $(G,M)$ a {\em reductive pair} if the $M$-stable subspace ${\mathfrak m}= {\rm Lie}(M)$ of ${\mathfrak g}$ has an $M$-stable complement.

\begin{prop}[{\cite[Thm.\ 3.35]{BMR}}]
\label{prop:sep_red_pair}
 Let $M$ be a connected reductive subgroup of $G$ such that $(G,M)$ is a reductive pair.  Let $H$ be a subgroup of $M$.  If $H$ is $G$-cr and separable in $G$ then $H$ is $M$-cr and separable in $M$.
\end{prop}

\begin{proof}
 (Sketch:) The proof uses a beautiful geometric argument due to Richardson.  We assume $H$ is topologically finitely generated: say, $H= {\mathcal G}(h_1,\ldots, h_m)$.  Since $H$ is $G$-cr, $G\cdot (h_1,\ldots, h_m)$ is a closed subset of $G^m$ (Theorem~\ref{thm:Gcr_crit}).  Now consider $G\cdot (h_1,\ldots, h_m)\cap M^m$: this is closed and is a union of certain $M$-conjugacy classes, one of which is $M\cdot (h_1,\ldots, h_m)$.  To show $H$ is $M$-cr, it is enough by Theorem~\ref{thm:Gcr_crit} to prove that $M\cdot (h_1,\ldots, h_m)$ is closed.
 
 We do this by studying the tangent space to $G\cdot (h_1,\ldots, h_m)\cap M^m$ at the point $(h_1,\ldots, h_m)$.  Given $g\in G$, define $R_g\colon G\ra G$ by $R_g(g')= g'g$.  We may identify the tangent space $T_gG$ with $T_1G= {\mathfrak g}$ via the derivative $d_gR_{g^{-1}}$.  Hence we may identify $T_{(h_1,\ldots, h_m)}G^m$ with ${\mathfrak g}^m$.  By the separability assumption on $H$ and the derivative criterion for separability, the tangent space $T_{(h_1,\ldots, h_m)}(G\cdot (h_1,\ldots, h_m))$ is given by
 \begin{equation}
 \label{eqn:Gorbit}
  T_{(h_1,\ldots, h_m)}(G\cdot (h_1,\ldots, h_m))= \{(x- h_1\cdot x,\ldots, x- h_m\cdot x)\mid x\in {\mathfrak g}\}\subseteq {\mathfrak g}^m
 \end{equation}
 (where the $\cdot$ denotes the adjoint action of $G$ on ${\mathfrak g}$).  Likewise,
 \begin{equation}
 \label{eqn:Morbit}
  T_{(h_1,\ldots, h_m)}(M\cdot (h_1,\ldots, h_m))= \{(x- h_1\cdot x,\ldots, x- h_m\cdot x)\mid x\in {\mathfrak m}\}\subseteq {\mathfrak g}^m.
 \end{equation}
It follows from the derivative criterion for separability that $H$ is separable in $M$.  Now $G\cdot (h_1,\ldots, h_m)\cap M^m$ is a closed subvariety of $G^m$ as $G\cdot (h_1,\ldots, h_m)$ is closed, and Eqn.\ (\ref{eqn:Morbit}) shows that the orbits $M\cdot y$ for $y\in G\cdot (h_1,\ldots, h_m)\cap M^m$ all have the same dimension---namely, ${\rm dim}(M)- {\rm dim}(C_M(H))$. (Recall that the closure of an orbit is the union of that orbit with some orbits of smaller dimension.)  It follows that $M\cdot y$ is closed for every $y\in G\cdot (h_1,\ldots, h_m)\cap M^m$, as required.
\end{proof}

\begin{ex}[{\cite[Ex.\ 5.7]{BHMR_ss}}]
 Let $M= {\rm SL}_2(k)$ or ${\rm PGL}_2(k)$ and suppose $p= 2$.  Let $H= N_M(T)$, where $T$ is a maximal torus of $M$; it is straightforward to show that $H$ is not separable in $M$.  Proposition~\ref{prop:sep_red_pair} implies that there does not exist an embedding of $M$ in $G= {\rm GL}_n(k)$ for any $n$ such that $(G,M)$ is a reductive pair (recall that any subgroup of ${\rm GL}_n(k)$ is separable).
\end{ex}

\begin{rem}
 We note one further consequence of the arguments in the proof of Proposition~\ref{prop:sep_red_pair}.  A short calculation using our assumption that $(G,M)$ is a reductive pair, together with Eqns.\ (\ref{eqn:Gorbit}) and (\ref{eqn:Morbit}), shows that
$$ T_{(h_1,\ldots, h_m)}(G\cdot (h_1,\ldots, h_m)\cap M^m)   = T_{(h_1,\ldots, h_m)}(M\cdot (h_1,\ldots, h_m)), $$
which implies that $M\cdot (h_1,\ldots, h_m)$ is an open subset of $G\cdot (h_1,\ldots, h_m)\cap M^m$.  But we saw above that $M\cdot (h_1,\ldots, h_m)$ is also a closed subset of $G\cdot (h_1,\ldots, h_m)\cap M^m$, so $M\cdot (h_1,\ldots, h_m)$ must be a union of irreducible components of $G\cdot (h_1,\ldots, h_m)\cap M^m$.  It follows that $G\cdot (h_1,\ldots, h_m)\cap M^m$ is a {\bf finite} union of $M$-conjugacy classes.

This property can fail without the hypotheses of Proposition~\ref{prop:sep_red_pair}.  Let $G$ and $M$ be as in Example~\ref{ex:BMRT_TAMS}; then $(G,M)$ is a reductive pair.  Recall that there is a subgroup $H$ of $M$ such that $H$ is $G$-cr but not $M$-cr.  A related construction yields a pair $(m_1,m_2)\in M^2$ such that $G\cdot (m_1,m_2)\cap M^2$ is an infinite union of $M$-conjugacy classes.  Set $\widehat{H}:= {\mathcal G}(m_1,m_2)$.  We deduce from the discussion above that neither $H$ nor $\widehat{H}$ is separable in $G$; one can check this by explicit calculation \cite[Sec.\ 7]{BMRT}. 
\end{rem}

\begin{cor}[{\cite[Ex.\ 3.37]{BMR}}]
 Let $G$ be a simple group of exceptional type, and suppose $p$ is good for $G$.  Let $H$ be a subgroup of $G$ such that $H$ acts semisimply on ${\mathfrak g}$.  Then $H$ is $G$-cr.
\end{cor}

\begin{proof}
 We apply Proposition~\ref{prop:sep_red_pair} to the inclusion $H\leq G\leq {\rm GL}({\mathfrak g})$.  Define a symmetric bilinear form ${\mathfrak B}$ on ${\rm Lie}({\rm GL}({\mathfrak g}))= {\rm End}({\mathfrak g})$ by ${\mathfrak B}(X,Y)= {\rm trace}(XY)$.  It is easily checked that ${\mathfrak B}$ is nondegenerate and ${\rm GL}({\mathfrak g})$-invariant.  The hypothesis on $p$ implies that the restriction of ${\mathfrak B}$ to ${\mathfrak g}$ is a nonzero multiple of the Killing form on ${\mathfrak g}$.  Let ${\mathfrak d}$ be the orthogonal complement to ${\mathfrak g}$ in ${\rm End}({\mathfrak g})$ with respect to ${\mathfrak B}$: then ${\rm dim}({\mathfrak d})+ {\rm dim}({\mathfrak g})= {\rm dim}({\rm End}({\mathfrak g}))$ as ${\mathfrak B}$ is nondegenerate, and ${\mathfrak d}\cap {\mathfrak g}= 0$ as the Killing form on ${\mathfrak g}$ is nondegenerate.  It follows that ${\mathfrak d}$ is a $G$-stable complement to ${\mathfrak g}$ in ${\rm End}({\mathfrak g})$, so $({\rm GL}({\mathfrak g}), G)$ is a reductive pair.

 Now $H$ is ${\rm GL}({\mathfrak g})$-cr as $H$ acts semisimply on ${\mathfrak g}$.  Since any subgroup of ${\rm GL}({\mathfrak g})$ is separable, it follows from Proposition~\ref{prop:sep_red_pair} that $H$ is $G$-cr.
\end{proof}

\noindent There is a similar result for simple groups of classical type.

The next result shows we get the same outcome under slightly different hypotheses.

\begin{prop}[{\cite[Thm.\ 3.46]{BMR}}]
\label{prop:BMR_3.46}
 Suppose $H$ is a separable subgroup of $G$ and $H$ acts semisimply on ${\mathfrak g}$.  Then $H$ is $G$-cr.
\end{prop}

\begin{proof}
 To simplify notation, we assume that the adjoint representation of $G$ yields an embedding of $G$ in ${\rm GL}({\mathfrak g})$; then we can regard $H$ as a subgroup of ${\rm GL}({\mathfrak g})$.  We can assume that $H= {\mathcal G}({\mathbf h})$ for some tuple ${\mathbf h}$.  Suppose $H$ is not $G$-cr.  By Theorem~\ref{thm:Gcr_crit}, we can choose $\lambda\in Y(G)$ such that ${\mathbf h}':= \lim_{a\to 0} \lambda(a)\cdot {\mathbf h}$ exists and does not belong to $G\cdot {\mathbf h}$.  Set $H'= {\mathcal G}({\mathbf h}')$.  Since ${\mathbf h}'$ belongs to $\ovl{G\cdot {\mathbf h}}$ but not to $G\cdot {\mathbf h}$, we have ${\rm dim}(G\cdot {\mathbf h}')< {\rm dim}(G\cdot {\mathbf h})$, which implies that ${\rm dim}(G_{{\mathbf h}'})> {\rm dim}(G_{\mathbf h})$.  Now $G_{{\mathbf h}'}= C_G(H')$ and $G_{\mathbf h}= C_G(H)$, so we get ${\rm dim}(C_G(H'))> {\rm dim}(C_G(H))$.  As $H$ is separable in $G$, we deduce that
 \begin{equation}
 \label{eqn:dim_ineq}
  {\rm dim}({\mathfrak c}_{\mathfrak g}(H'))\geq {\rm dim}(C_G(H'))> {\rm dim}(C_G(H))= {\rm dim}({\mathfrak c}_{\mathfrak g}(H)).
 \end{equation}
 
 By hypothesis, $H$ is ${\rm GL}({\mathfrak g})$-cr.  It follows from Theorem~\ref{thm:Gcr_crit}---this time applied to ${\rm GL}({\mathfrak g})$---that ${\rm GL}({\mathfrak g})\cdot {\mathbf h}$ is closed, so ${\mathbf h'}$ is ${\rm GL}({\mathfrak g})$-conjugate to ${\mathbf h}$.  Hence $H$ and $H'$ are ${\rm GL}({\mathfrak g})$-conjugate.  Now ${\mathfrak c}_{\mathfrak g}(H)$ (resp., ${\mathfrak c}_{\mathfrak g}(H')$) is precisely the fixed-point space of $H$ (resp., $H'$) in ${\mathfrak g}$, so we must have ${\rm dim}({\mathfrak c}_{\mathfrak g}(H))= {\rm dim}({\mathfrak c}_{\mathfrak g}(H'))$.  But this contradicts (\ref{eqn:dim_ineq}).  We deduce that $H$ must be $G$-cr after all.
\end{proof}

\noindent {\bf Open Problem:} Does Proposition~\ref{prop:BMR_3.46} hold without the separability hypothesis on $H$?  See \cite[Sec.\ 4]{BMRT} for further discussion.

\section{Optimal destabilising parabolic subgroups}
\label{sec:opt}

The parabolic subgroup ${\mathcal P}(U)$ that we obtained from the Borel-Tits construction has special properties: it is a witness that $U$ is not $G$-cr and it contains $N_G(U)$.  In this section we establish the existence of parabolic subgroups with similar properties in a wider context.

\begin{thm}
\label{thm:opt}
 Let $H$ be a subgroup of $G$ such that $H$ is not $G$-cr.  Then there is a parabolic subgroup $P_{\rm opt}(H)$ of $G$ such that $P_{\rm opt}(H)$ is a witness that $H$ is not $G$-cr and $N_G(H)\leq P_{\rm opt}(H)$.
\end{thm}

\noindent In fact, $P_{\rm opt}(H)$ is stabilised by any automorphism of $G$ that stabilises $H$.

There may exist several parabolic subgroups of $G$ with the desired properties, but we have a particular construction in mind.  The ``opt'' subscript is short for ``optimal''---we find $P_{\rm opt}(H)$ by optimising a convex real-valued function on the space of cocharacters of $G$.  We call $P_{\rm opt}(H)$ the {\em optimal destabilising parabolic subgroup} for $H$.

Theorem~\ref{thm:opt} is a special case of a more general theorem from GIT which we call the Hilbert-Mumford-Kempf-Hesselink-Rousseau Theorem (cf.\ \cite{kempf}, \cite{hesselink}, \cite{rousseau}).  Given a $G$-variety $X$ and a point $x\in X$ such that $G\cdot x$ is not closed, we can construct an {\em optimal destabilising cocharacter} $\lambda_{\rm opt}(x)$ such that $x':= {\rm lim}_{a\to 0} \lambda_{\rm opt}(x)(a)\cdot x$ exists and $G\cdot x'$ is closed.  Then $x'\not\in G\cdot x$; roughly speaking, we can think of $\lambda_{\rm opt}(x)$ as the cocharacter that takes $x$ outside the orbit $G\cdot x$ and into the closed orbit $G\cdot x'$ ``as quickly as possible''.  The cocharacter $\lambda_{\rm opt(x)}$ is unique up to $R_u(P_{\lambda_{\rm opt}(x)})$-conjugacy (modulo a normalisation condition which we discuss below); hence the parabolic subgroup $P_{\rm opt}(x):= P_{\lambda_{\rm opt}(x)}$ is uniquely determined.  Moreover, the construction is natural in an appropriate sense: for any $g\in G$, $P_{\rm opt}(g\cdot x)= gP_{\rm opt}(x)g^{-1}$.  In particular, $G_x$ normalises $P_{\rm opt}(x)$, so $G_x\leq P_{\rm opt}(x)$.

Theorem~\ref{thm:opt} is a consequence of this construction.  For suppose $H\leq G$ is not $G$-cr.  We assume as usual that $H= {\mathcal G}({\mathbf h})$ for some tuple ${\mathbf h}\in G^m$.  By Theorem~\ref{thm:Gcr_crit}, $G\cdot {\mathbf h}$ is not closed.  We can associate to ${\mathbf h}$ the optimal destabilising parabolic subgroup $P_{\rm opt}({\mathbf h})$ from the Hilbert-Mumford-Kempf-Hesselink-Rousseau Theorem, and we set $P_{\rm opt}(H):= P_{\rm opt}({\mathbf h})$.  Then $G_{\mathbf h}= C_G(H)$ is contained in $P_{\rm opt}(H)$.

There is one problem: we do not know whether $P_{\rm opt}(H)$ is dependent on the choice of ${\mathbf h}$.  In particular, if $g\in G$ normalises $H$ but does not centralise $H$ then conjugation by $g$ takes the generating tuple ${\mathbf h}$ to a different generating tuple, so we cannot conclude {\em a priori} that $g$ normalises $P_{\rm opt}(H)$.  One can overcome this difficulty using Hesselink's notion of ``uniform $S$-instability": rather than applying a cocharacter $\lambda$ to the single point ${\mathbf h}$, we apply it to the entire set $H^m$.  One can construct an optimal destabilising cocharacter and parabolic subgroup as before; if $g\in G$ normalises $H$ then $g$ stabilises $H^m$, and it follows that $g$ normalises $P_{\rm opt}(H)$.  This gives $N_G(H)\leq P_{\rm opt}(H)$, as required.  See \cite[Sec.\ 5.2]{BMRT_git} for details.  We will ignore this issue and just concentrate on the simpler version of the construction.

\subsection{The construction}

Now we explain how to obtain the cocharacter $\lambda_{\rm opt}(x)$ described in the Hilbert-Mumford-Kempf-Hesselink-Rousseau Theorem, following the treatment of Kempf \cite{kempf}.
We restrict ourselves to a special case which still illustrates the main ideas.  Let $V$ be a $G$-module.  Let us consider {\em unstable} points in $V$: that is, points $v\in V$ such that the origin 0 belongs to $\ovl{G\cdot v}$.  Fix an unstable point $0\neq v\in V$; we will explain how to define $\lambda_{\rm opt}(v)$.  By the Hilbert-Mumford Theorem, $\lim_{a\to 0} \lambda(a)\cdot v= 0$ for some $\lambda\in Y(G)$.  Choose a maximal torus $T$ of $G$ such that $\lambda\in Y(T)$.  Then $\langle \lambda, \chi\rangle> 0$ for all $\chi\in {\rm supp}_T(v)$ by Example~\ref{ex:module_lim}.  Write ${\rm supp}_T(v)= \{\chi_1,\ldots, \chi_t\}$.  Then $v$ can be written as a sum of nonzero weight vectors
$v_1,\ldots, v_t$ corresponding to the weights $\chi_1,\ldots, \chi_t$, respectively, and we have
\begin{equation}
\label{eqn:speed}
 \lambda(a)\cdot v= \lambda(a)\cdot (v_1+\cdots + v_t)= a^{n_1}v_1+\cdots +a^{n_t}v_t
\end{equation}
for all $a\in k^*$, where $n_i:= \langle \lambda, \chi_i\rangle> 0$ for $1\leq i\leq t$.  Intuitively, the speed at which $\lambda(a)\cdot v$ approaches 0 is determined by the smallest of the $n_i$.

This motivates the following definition.

\begin{defn}
\label{defn:numerical}
 Let $V$, $v$, $T$ and $\chi_1,\ldots, \chi_t$ be as above.  Define $\mu_{v,T}\colon Y(T)\ra {\mathbb Z}$ by
 \begin{equation}
\label{eqn:linmin}
 \mu_{v,T}(\lambda)= \min_{1\leq i\leq t} \psi_i(\lambda),
\end{equation}
where $\psi_i\colon Y(T)\ra {\mathbb Z}$ is given by $\psi_i(\lambda)= \langle \lambda, \chi_i\rangle$.

 We see that $\lim_{a\to 0} \lambda(a)\cdot v$ exists if and only if $\mu_{v,T}(\lambda)\geq 0$, and $\lim_{a\to 0} \lambda(a)\cdot v= 0$ if and only if $\mu_{v,T}(\lambda)> 0$.  Given $g\in G$, we have
\begin{equation}
\label{eqn:mu_nat}
 \mu_{g\cdot v, gTg^{-1}}(g\cdot \lambda)= \mu_{v,T}(\lambda)
\end{equation}
(exercise).
\end{defn}

In fact, we can show that the value of $\mu_{v,T}(\lambda)$ doesn't depend on the choice of $T$.  To see this, recall that the decomposition of $V$ into weight spaces works for any torus $S$ of $G$, not just for a maximal torus.  In particular, it works for $S= {\rm Im}(\lambda)$.  If $\zeta$ is a weight of $V$ with respect to $S$ then $V_\zeta$ is $T$-stable, so it splits into a direct sum of weight spaces $V_\chi$ for $V$ with respect to $T$, and we have $\langle \lambda, \zeta\rangle= \langle \lambda, \chi\rangle$ for every $\chi$ that appears in this sum.  The assertion now follows.

We define $\mu_v\colon Y(G)\ra {\mathbb Z}$ by $\mu_v(\lambda)= \mu_{v,T}(\lambda)$, where $T$ is any maximal torus of $G$ such that $\lambda\in Y(T)$.  We call $\mu_v$ the {\em numerical function} associated to $v$.

\begin{lem}
\label{lem:Pinvce}
 Let $V$ and $v$ be as above, let $\lambda\in Y(G)$ and let $u\in R_u(P_\lambda)$.  Then $\mu_v(\lambda)= \mu_v(u\cdot \lambda)$.
\end{lem}

\begin{proof}
 By (\ref{eqn:mu_nat}), $\mu_v(u\cdot \lambda)= \mu_{u^{-1}\cdot v}(\lambda)$, so it's enough to show that $\mu_{u^{-1}\cdot v}(\lambda)= \mu_v(\lambda)$.  Fix a maximal torus $T$ of $G$ such that $\lambda\in Y(T)$; we show that $\mu_{u^{-1}\cdot v,T}(\lambda)= \mu_{v,T}(\lambda)$.  Let $n= \mu_{v,T}(\lambda)$; then there exists at least one $\widetilde{\chi}\in {\rm supp}_T(v)$ such that $\langle \lambda, \widetilde{\chi}\rangle= n$, and $\langle \lambda, \chi\rangle\geq n$ for all $\chi\in {\rm supp}_T(v)$.  Lemma~\ref{lem:wts_increase} implies that $u^{-1}\cdot v= v+ w$ for some $w\in V$ such that $\langle \lambda, \chi'\rangle> n$ for all $\chi'\in {\rm supp}_T(w)$.  It follows that $\mu_{u^{-1}\cdot v,T}(\lambda)= \mu_{v,T}(\lambda)$, as required.
\end{proof}

It is convenient below to work with real vector spaces rather than ${\mathbb Z}$-modules.  We may regard $Y(T)$ as an integer lattice inside the real vector space $Y_{\mathbb R}(T):= Y(T)\otimes_{\mathbb Z} {\mathbb R}$.  Just as we can form $Y(G)$ by glueing together the pieces $Y(T)$, we can form a space $Y_{\mathbb R}(G)$ by glueing together the pieces $Y_{\mathbb R}(T)$.  (This construction is not entirely straightforward---cf.\ \cite[Sec.\ 2]{BMR_sTCC}---but we omit the details.)  We may regard $Y(G)$ as a subset of $Y_{\mathbb R}(G)$, and the $G$-action on $Y(G)$ extends to a $G$-action on $Y_{\mathbb R}(G)$ in a natural way.  We can extend the functions $\psi_i$ from Definition~\ref{defn:numerical} to ${\mathbb R}$-linear functions $\psi_i\colon Y_{\mathbb R}(T)\ra {\mathbb R}$.  This allows us to extend $\mu_{v,T}$ to a function from $Y_{\mathbb R}(T)$ to ${\mathbb R}$ via (\ref{eqn:linmin}), and one can show we get a well-defined function $\mu_v\colon Y_{\mathbb R}(G)\ra {\mathbb R}$ (note that $Y_{\mathbb R}(G)$ is the union of all the $Y_{\mathbb R}(T)$).

We need one more ingredient before we prove the existence of $\lambda_{\rm opt}(x)$.  If we multiply $\lambda$ in (\ref{eqn:speed}) by a positive integer $m$ then the integers $n_i$ are replaced by $mn_i$.  To make sense of the intuitive idea that $\lambda_{\rm opt}$ is the cocharacter that ``takes $v$ to 0 as fast as possible'', we need some way to measure the size of $\lambda$.  We do this by means of a {\em length function}: this is a $G$-invariant function $|\!|\cdot |\!|\colon Y_{\mathbb R}(G)\ra {\mathbb R}$, $\lambda\mapsto |\!|\lambda |\!|$, such that the restriction of $|\!|\cdot |\!|$ to each vector space $Y_{\mathbb R}(T)$ is the norm arising from a nondegenerate symmetric ${\mathbb Z}$-valued bilinear form on $Y(T)$.  (So $|\!|\lambda |\!|\geq 0$ for all $\lambda\in Y_{\mathbb R}(G)$, with equality if and only if $\lambda= 0$, and $|\!|c\lambda |\!|= c|\!|\lambda |\!|$ for all $\lambda\in Y_{\mathbb R}(G)$ and all $c\in {\mathbb R}_+$.)  Below we fix, once and for all, a choice of length function.

We define $f_v\colon Y_{\mathbb R}(G)\backslash\{0\}\ra {\mathbb R}$ by
\begin{equation}
\label{eqn:normalised_numerical}
 f_v(\lambda)= \frac{\mu_v(\lambda)}{|\!|\lambda |\!|}.
\end{equation}
If $T$ is a maximal torus of $G$ then we denote by $f_{v,T}$ the restriction of $f_v$ to $Y_{\mathbb R}(T)\backslash\{0\}$.  It follows from Lemma~\ref{lem:Pinvce} and the conjugation-invariance of $|\!|\cdot |\!|$ that
\begin{equation}
\label{eqn:f_Pinvce}
 f_v(u\cdot \lambda)= f_v(\lambda)\ \mbox{for all $\lambda\in Y(G)\backslash \{0\}$ and all $u\in R_u(P_\lambda)$}.
\end{equation}

We have
\begin{equation}
\label{eqn:invce}
 f_{v,T}(c\lambda)= f_{v,T}(\lambda)
\end{equation}
for any $\lambda\in Y_{\mathbb R}(T)\backslash \{0\}$ and any $c\in {\mathbb R}_+$---the factors of $c$ that appear in the numerator and the denominator of (\ref{eqn:normalised_numerical}) cancel each other out.

\begin{lem}
\label{lem:loc_max}
 Fix a maximal torus $T$ of $G$.  There exists $\lambda_{v,T}\in Y(T)\backslash \{0\}$ such that $f_{v,T}$---as a function on $Y_{\mathbb R}(T)$---attains its maximum value $C_{v,T}$ at $\lambda_{v,T}$.  Moreover, $\lambda_{v,T}$ is unique subject to the condition that $|\!|\lambda_{v,T}|\!|$ is minimal.
\end{lem}

\noindent The uniqueness condition needs a word of explanation.  The proof below shows that the set
$$ \{\lambda\in Y_{\mathbb R}(T)\backslash \{0\}\mid f_{v,T}(\lambda)= C_{v,T}\} $$
is a ray ${\mathcal R}$: that is, it has the form $\{c\lambda_1\mid c\in {\mathbb R}_+\}$ for some $0\neq \lambda_1\in Y_{\mathbb R}(T)$.  The proof also shows that ${\mathcal R}$ contains at least one element of $Y(T)$.  As $Y(T)$ is a lattice in $Y_{\mathbb R}(T)$, there is a unique element $\lambda_{v,T}\in Y(T)$ of ${\mathcal R}$ that is closest to the origin.

\begin{proof}
 For simplicity, we assume $t= 1$ in (\ref{eqn:linmin}): so $f_{v,T}$ has the form $\displaystyle f_{v,T}(\lambda)= \frac{\psi(\lambda)}{|\!|\lambda|\!|}$ for some linear function $\psi\colon Y_{\mathbb R}(T)\ra {\mathbb R}$.  Let $S$ be the unit sphere in $Y_{\mathbb R}(T)$ (with respect to $|\!|\cdot |\!|$).  Then $f_{v,T}|_S$ is a continuous function on the compact set $S$, so it attains a maximum value $C_{v,T}$---say, at $\lambda_1\in S$---and (\ref{eqn:invce}) implies that $C_{v,T}$ is the maximum value attained by $f_{v,T}$ on the whole of $Y_{\mathbb R}(T)\backslash \{0\}$.  Suppose $f_{v,T}(\lambda_2)= f_{v,T}(\lambda_1)$ for some $\lambda_2\in S$ with $\lambda_2\neq \lambda_1$.  Choose any $c\in (0,1)$; set $\lambda_3= c\lambda_1+ (1- c)\lambda_2$.  Then $\psi(\lambda_1)= \psi(\lambda_2)= C_{v,T}$, so $\psi(\lambda_3)= C_{v,T}$ by linearity of $\psi$; but $|\!|\lambda_3|\!|< 1$ since $S$ is convex.  This gives $f_{v,T}(\lambda_3)> C_{v,T}$, a contradiction.  We deduce that the set of points in $Y_{\mathbb R}(T)\backslash \{0\}$ where $f_{v,T}$ attains its maximum value $C_{v,T}$ is precisely the ray ${\mathcal R}$ through $\lambda_1$.
 
 Because $|\!|\cdot |\!|$ and the linear functions $\psi_i$ are defined over ${\mathbb Z}$ in a suitable sense, it can be shown that ${\mathcal R}$ contains a point from $Y(T)$; we omit the details.  The uniqueness of $\lambda_{v,T}$ now follows from the paragraph before the proof.
\end{proof}

We can now state and prove our main theorem.

\begin{thm}
 The function $f_v$ attains its maximum value $C$ at some $\lambda_{\rm opt}= \lambda_{\rm opt}(v)\in Y(G)$.  Moreover, $\lambda_{\rm opt}$ is unique up to $R_u(P_{\lambda_{\rm opt}})$-conjugacy, subject to the condition that $|\!|\lambda_{\rm opt}|\!|$ is minimal.
\end{thm}

\begin{proof}
 First we show that the set $\{C_{v,T}\mid \mbox{$T$ is a maximal torus of $G$}\}$ is finite.  Eqn.\ (\ref{eqn:mu_nat}) implies that $\mu_{v,gTg^{-1}}(\lambda)= \mu_{g^{-1}\cdot v, T}(g^{-1}\cdot \lambda)$ for every maximal torus $T$, every $\lambda\in Y(T)$ and every $g\in G$.  It follows that $C_{v,gTg^{-1}}= C_{g^{-1}\cdot v, T}$ for all $g\in G$.  But $\Phi_T(V)$ (the set of weights of $T$ on $V$) is finite, so the set of subsets $\{{\rm supp}_T(g^{-1}\cdot v)\,|\,g\in G\}$ of $\Phi_T(V)$ is finite, and our claim follows.
 
 So let $C$ be the largest of the $C_{v,T}$.  Clearly $C$ is the maximum value of $f_v$, and it is attained at some $0\neq \lambda_{\rm opt}\in Y(G)$ of minimum length.  Let $T_1$ and $T_2$ be maximal tori of $G$ and let $\lambda_1:= \lambda_{T_1}\in Y(T_1)$ and $\lambda_2:= \lambda_{T_2}\in Y(T_2)$ be the cocharacters provided by Lemma~\ref{lem:loc_max}.  Set $P_1= P_{\lambda_1}$ and $P_2= P_{\lambda_2}$.  Suppose $f_v(\lambda_1)= f_v(\lambda_2)= C$.  It is enough to prove that $\lambda_2\in R_u(P_1)\cdot \lambda_1$.
 
 Recall from Section~\ref{sec:alggps} that $P_1\cap P_2$ contains a maximal torus $T$ of $G$; clearly $T$ is also a maximal torus of both $P_1$ and $P_2$.  As the maximal tori $T_1$ and $T$ of $P_1$ are conjugate, we have $g_1\cdot \lambda_1\in Y(T)$ for some $g_1\in P_1$.  But $\lambda_1$ is fixed by $L_{\lambda_1}$ (a Levi subgroup of $P_1$), so $u_1\cdot \lambda_1\in Y(T)$ for some $u_1\in R_u(P_1)$.  Now $f_v(u_1\cdot \lambda_1)= f_v(\lambda_1)= C$ by (\ref{eqn:f_Pinvce}).  It follows that $C= C_{v,T}$ and $u_1\cdot \lambda_1= c\lambda_{v,T}$ for some $c\in {\mathbb R}+$, where $\lambda_{v,T}$ is as in Lemma~\ref{lem:loc_max}.  Minimality of $|\!|\lambda_1|\!|$ and $|\!|\lambda_{v,T}|\!|$ implies that $c= 1$, so $u_1\cdot \lambda_1= \lambda_{v,T}$.
 
 By the same argument, $u_2\cdot \lambda_2= \lambda_{v,T}$ for some $u_2\in R_u(P_2)$.  But then $P_1= P_{\lambda_{v,T}}= P_2$, so $R_u(P_1)= R_u(P_2)$ and we get $\lambda_2= u_2^{-1}u_1\cdot \lambda_1\in R_u(P_1)\cdot \lambda_1$.  This completes the proof.
\end{proof}

\begin{rem}
 The constructions described above are well-behaved under conjugation by $G$: cf.\ (\ref{eqn:mu_nat}).  It follows that $P_{\rm opt}(g\cdot v)= gP_{\rm opt}(v)g^{-1}$ for any $g\in G$.  We leave the details to the reader.
\end{rem}

\subsection{Applications to $G$-complete reducibility}
\label{sec:opt_applns}

We spend the rest of this section deriving some consequences of Theorem~\ref{thm:opt} for $G$-complete reducibility.

\begin{thm}[{\cite[Thm.\ 3.10]{BMR}}]
\label{thm:Clifford}
 Let $H\leq G$ be $G$-cr and let $N$ be a normal subgroup of $H$.  Then $N$ is $G$-cr.
\end{thm}

\begin{proof}
 Suppose $N$ is not $G$-cr.  Then $P_{\rm opt}(N)$ is a witness that $N$ is not $G$-cr, and $H\leq N_G(N)\leq P_{\rm opt}(N)$.  Since $N$ is not contained in any Levi subgroup of $P_{\rm opt}(N)$, the larger group $H$ cannot be, either.  But this contradicts our assumption that $H$ is $G$-cr.  We conclude that $N$ is $G$-cr after all.
\end{proof}

\begin{rem}
 Let $G= {\rm GL}_n(k)$, and suppose $i\colon H\ra {\rm GL}_n(k)$ is a completely reducible embedding.  Clifford's Theorem says that the restriction of $i$ to a normal subgroup $N$ of $H$ is completely reducible.  Translating this into the language of $G$-complete reducibility, we see that if $H$ is a ${\rm GL}_n(k)$-cr subgroup of ${\rm GL}_n(k)$ then any normal subgroup of $H$ is ${\rm GL}_n(k)$-cr.  So Theorem~\ref{thm:Clifford} extends Clifford's Theorem to arbitrary $G$.
\end{rem}

\begin{prop}
\label{prop:normcent}
 Let $H$ be a $G$-cr subgroup of $G$.  Then $C_G(H)$ and $N_G(H)$ are $G$-cr.
\end{prop}

\begin{proof}
 For simplicity, we prove that $C_G(H)^0$ and $N_G(H)^0$ are $G$-cr under the assumption that $H$ is connected; the proof for $C_G(H)$ and $N_G(H)$ for general $H$ is completely analogous, but requires the formalism of $G$-complete reducibility for nonconnected reductive groups.  Recall from the discussion after Remark~\ref{rem:topfg} that $N_G(H)^0$ is reductive.  Let $P$ be a parabolic subgroup of $G$ that contains $N_G(H)^0$.  Then $N_G(H)^0$ contains $H$, so $P$ contains $H$, so some Levi subgroup $L$ of $P$ contains $H$, as $H$ is $G$-cr.  We can write $P= P_\lambda$ and $L= L_\lambda$ for some $\lambda\in Y(G)$.  Now $\lambda$ centralises $H$, so $\lambda$ is a cocharacter of $N_G(H)^0$.  We have $P_\lambda(N_G(H)^0)= P_\lambda\cap N_G(H)^0= N_G(H)^0$ since $N_G(H)^0\leq P_\lambda$.  But then $L_\lambda(N_G(H)^0)= N_G(H)^0$ (recall that a connected reductive group is a Levi subgroup of itself).  So $N_G(H)^0\leq L_\lambda$.
 
 This shows that $N_G(H)^0$ is $G$-cr.  It follows from Theorem~\ref{thm:Clifford} that $C_G(H)^0$ is $G$-cr, since $C_G(H)^0$ is normal in $N_G(H)^0$.
\end{proof}

\begin{cor}
 Let $H\leq G$.  Then $H$ is $G$-cr if and only if $N_G(H)$ is $G$-cr.
\end{cor}

\begin{proof}
 This follows from Proposition~\ref{prop:normcent} and Theorem~\ref{thm:Clifford}.
\end{proof}

\begin{prop}
\label{prop:cent}
 Let $S$ be a $G$-cr subgroup of $G$ and let $H\leq C_G(S)$ such that $H$ is $C_G(S)$-cr.  Then $H$ is $G$-cr.
\end{prop}


\begin{proof}
 As before, we replace $C_G(S)$ with $M:= C_G(S)^0$ to avoid having to deal with non-connected reductive groups.  Recall from Section~\ref{sec:Gcr_GIT} that $M$ is reductive.  Suppose $H$ is not $G$-cr.  Then $P_{\rm opt}(H)$ is a witness that $H$ is not $G$-cr and $N_G(H)\leq P_{\rm opt}(H)$; in particular, $S\leq P_{\rm opt}(H)$.  Since $S$ is $G$-cr, we can write $P_{\rm opt}(H)= P_\lambda$ for some $\lambda\in Y(G)$ such that $H\leq L_\lambda$.  Then $\lambda\in Y(M)$ and $H\leq P_\lambda(M)$.
 
 We claim that $P_\lambda(M)$ is a witness that $H$ is not $M$-cr.  Suppose $H$ is contained in a Levi subgroup of $P_\lambda(M)$.  Then $H\leq uL_\lambda(M)u^{-1}$ for some $u\in R_u(P_\lambda(M))$.  But then $H\leq uL_\lambda u^{-1}$, which is a Levi subgroup of $P_{\rm opt}(H)$ since $u\in R_u(P_{\rm opt}(H))$, so we get a contradiction.   This shows that $H$ is not $M$-cr.  But this contradicts our assumption on $H$.  We conclude that $H$ is $G$-cr after all.
\end{proof}

\begin{rem}
\label{rem:cent_cvse}
 It can be shown that if $S$ is linearly reductive then the converse to Proposition~\ref{prop:cent} also holds \cite[Cor.\ 3.21]{BMR}: so in this case, $H$ is $G$-cr if and only if $H$ is $C_G(S)$-cr (and likewise for $C_G(S)^0$).  This is false if $S$ is not linearly reductive, as we will see shortly.  In fact, one can show that $H$ is $C_G(S)$-cr if and only if $HS$ is $G$-cr \cite[Prop.\ 3.9]{BMR_commuting}.
 
 We mention a useful corollary in the linearly reductive case \cite[Ex.\ 3.23]{BMR}.  The group ${\rm SO}_n(k)$ sits inside ${\rm SL}_n(k)$.  Suppose $p\neq 2$ and let $H$ be a subgroup of ${\rm SO}_n(k)$.  Then $H$ is ${\rm SO}_n(k)$-cr if and only if $H$ is ${\rm SL}_n(k)$-cr.  To see this, observe that ${\rm SO}_n(k)$ is the fixed point set of the involution $\phi\in {\rm Aut}({\rm SL}_n(k))$ given by $\phi(A)= (A^T)^{-1}$.  The result now follows from (the non-connected version of) the previous paragraph applied to the (non-connected) reductive group $G:= S\ltimes {\rm SL}_n(k)$, where $S:= \langle \phi\rangle$---note that $S$ is linearly reductive as $p\neq 2$ and $C_G(S)^0= {\rm SO}_n(k)$.  By a similar argument, if $p\neq 2$ and $H$ is a subgroup of ${\rm Sp}_{2n}(k)$ then $H$ is ${\rm Sp}_{2n}(k)$-cr if and only if $H$ is ${\rm SL}_{2n}(k)$-cr.
\end{rem}

We finish the section by considering the following question.  Suppose $H_1$ and $H_2$ are commuting $G$-cr subgroups of $G$.  Is the product $H_1H_2$ also $G$-cr?  The answer is yes for $G= {\rm SL}_n(k)$ and $G= {\rm GL}_n(k)$, by an argument of Tange \cite[Lem.\ 4.5]{BMR_commuting}.  (Surprisingly, the following question seems to be open: if $G= {\rm SL}_n(k)$ or $G= {\rm GL}_n(k)$ and $H_1$ and $H_2$ are $G$-cr subgroups such that $H_1$ normalises $H_2$, must $H_1H_2$ also be $G$-cr?)  It follows from Remark~\ref{rem:cent_cvse} that the answer is yes for $G= {\rm SO}_n(k)$ and $G= {\rm Sp}_{2n}(k)$ if $p\neq 2$.  Using this fact together with detailed information due to Liebeck and Seitz about the subgroup structure of simple groups of exceptional type, Bate-Martin-R\"ohrle proved the following result.

\begin{prop}[{\cite[Thm.\ 1.3]{BMR_commuting}}]
 Let $G$ be connected and suppose $p$ is good for $G$ or $p> 3$.  Let $H_1$ and $H_2$ be connected $G$-cr subgroups of $G$ such that $H_1$ and $H_2$ commute.  Then $H_1H_2$ is $G$-cr.
\end{prop} 

\noindent But the answer to the question is no in general.  Liebeck has given an example with $p= 2$ and $G= {\rm Sp}_8(k)$ \cite[5.3]{BMR_commuting}; he found connected reductive subgroups $H_1$ and $H_2$ of $G$ such that $H_1H_2$ is not $G$-cr.  By Remark~\ref{rem:cent_cvse}, $H_1$ is not $C_G(H_2)$-cr and $H_2$ is not $C_G(H_1)$-cr.  This gives a counter-example to the converse of Proposition~\ref{prop:cent}.

\section{Other topics}

We briefly discuss some other topics related to $G$-complete reducibility and its applications.

\subsection{Non-connected $G$}
\label{sec:nonconn}

Even if we are interested mainly in $G$-complete reducibility for connected reductive groups, we have seen in Section~\ref{sec:opt_applns} that we must sometimes deal with non-connected ones.  The basic idea is simple.  Let $G$ be a non-connected reductive group.  Given $\lambda\in Y(G)$, we define
$$ P_\lambda= \left\{g\in G\ \left| \lim_{a\to 0} \lambda(a)g\lambda(a)^{-1}\ \mbox{exists}\right\}\right., $$
just as before, and we call $P_\lambda$ a {\em Richardson parabolic subgroup} (or {\em R-parabolic subgroup} for short).  We define $L_\lambda= C_G({\rm Im}(\lambda))$ as before, and we call $L_\lambda$ a {\em Richardson Levi subgroup} (or {\em R-Levi subgroup} for short).  Now we define a subgroup $H$ of $G$ to be $G$-completely reducible just as in Definition~\ref{defn:Gcr}, but replacing parabolic subgroups and Levi subgroups with R-parabolic subgroups and R-Levi subgroups, respectively.  See \cite[Sec.\ 6]{BMR} for details.

One technical point: R-parabolic subgroups are not always self-normalising.  But this does not cause any serious problems, because of the following result.

\begin{prop}[{\cite[Prop.\ 5.4(a)]{martin}}]
 If $P$ is any parabolic subgroup of $G^0$ then $N_G(P)$ is an R-parabolic subgroup of $G$.
\end{prop}

An algorithm is given in \cite{BHMR_nonconn} for determining whether a subgroup of non-connected $G$ is $G$-cr by reducing to the case of connected reductive groups.

\subsection{Non-algebraically closed fields}

In this section we take $k$ to be an arbitrary field of characteristic $p> 0$, with algebraic closure $\ovl{k}$.  We take the point of view adopted in Borel's book \cite{borel}: we regard a variety $X$ defined over $k$ (a $k$-variety) as a $\ovl{k}$-variety together with a choice of $k$-structure.  If $Y$ is a closed subvariety of $X$ then we call $Y$ a $k$-subvariety of $X$ if $Y$ is defined over $k$.

Let $G$ be a connected reductive group defined over $k$ and let $H$ be a $k$-subgroup of $G$.  We say that $H$ is {\em $G$-cr over $k$} if whenever $H$ is contained in a $k$-parabolic subgroup $P$ of $G$, $H$ is contained in a $k$-Levi subgroup $L$ of $P$.  If $G= {\rm GL}_n(k)$ or ${\rm SL}_n(k)$ then we have the same characterisation as before: $H$ is $G$-completely reducible over $k$ if and only if the inclusion of $H$ in $G$ is a completely reducible representation (over $k$).

If $k'/k$ is an algebraic field extension then we can extend scalars and regard $H$ as a $k'$-subgroup of the connected reductive $k'$-group $G$, and we can ask whether $H$ is $G$-cr over $k'$.  In particular, we can ask whether $H$ is $G$-cr over $\ovl{k}$: this is equivalent to saying that $H$ is $G$-cr in our original sense.  If $k$ is perfect then the theory of $G$-complete reducibility over $k$ is similar to the algebraically closed case: for instance, one can show that $H$ is $G$-cr over $k$ if and only if it is $G$-cr \cite[Thm.\ 5.8]{BMR}.  The theory is much more complicated, however, if $k$ is not perfect.  Using Weil restriction, one can give an example of a subgroup $H$ such that $H$ is $G$-cr over $k$ but not $G$-cr.  An example with $H$ $G$-cr but not $G$-cr over $k$ is given in \cite[Ex.\ 7.22]{BMRT}; this is closely related to Example~\ref{ex:BMRT_TAMS}.

  We have already mentioned an open problem for non-algebraically closed fields in connection with Theorem~\ref{thm:mainconjalgclsd}.  Here is another.\medskip\\
\noindent {\bf Open Problem:} Suppose $H$ is $G$-cr over $k$.  Must $C_G(H)$ be $G$-cr over $k$?\medskip\\
One complication here is that $C_G(H)$ need not even be defined over $k$, so the problem has to be formulated carefully.  See \cite{uchiyama} for further discussion.

\subsection{K\"ulshammer's question}

Let $F$ be a finite group with Sylow $p$-subgroup $F_p$.  K\"ulshammer asked the following question \cite[Sec.\ 2]{kuls}.

\begin{qn}
\label{qn:kukshammer}
 Fix a homomorphism $\sigma\colon F_p\ra G$.  Is it true that there are only finitely many conjugacy classes of homomorphisms $\rho\colon F\ra G$ such that $\rho|_{F_p}$ is conjugate to $\sigma$?
\end{qn}

\noindent K\"ulshammer himself showed the answer is yes if $G= {\rm GL}_n(k)$ or ${\rm SL}_n(k)$ using simple representation-theoretic ideas.  Slodowy showed the answer is yes if $p$ is good for $G$ \cite[I.5, Thm.\ 3]{slodowy}.  The answer in general, however, is no: there is a counter-example---closely related (yet again!) to Example~\ref{ex:BMRT_TAMS}---with $p= 2$ and $G$ simple of type $G_2$ \cite{BMR_kuls}.  This counter-example can be interpreted in terms of the non-abelian cohomology discussed in Section~\ref{sec:history}.   Uchiyama has similar examples \cite[Sec.\ 6]{uchiyama3}.\bigskip\\
\noindent {\bf Open Problem:} All known counter-examples to Question~\ref{qn:kukshammer}---including one of Cram involving a non-reductive group $G$ \cite{cram}---are for $p= 2$.  Is there a counter-example with $p$ odd?\bigskip\\
The connection with $G$-complete reducibility is as follows: since there are---by Theorem~\ref{thm:fin_cr_rep}---only finitely many conjugacy classes of homomorphisms $\rho\colon F\ra G$ with $G$-cr image, a counter-example to Question~\ref{qn:kukshammer} must involve non-$G$-cr subgroups of $G$.

\subsection{Finite groups of Lie type}

We give an application of $G$-complete reducibility to simple groups of Lie type.  Suppose $G$ is a simple (algebraic) group of adjoint type.  Let $\sigma\colon G\ra G$ be a Frobenius map; then the fixed point subgroup $G^\sigma$ is a finite group of Lie type.  We have the following result.

\begin{prop}[{\cite[Prop.\ 2.2]{LMS}}]
\label{prop:Gcr_not_Gcr}
 Let $F$ be a finite $\sigma$-stable subgroup of $G$.  Then at least one of the following holds:\smallskip\\
 (a) $F$ is $G$-cr;\smallskip\\
 (b) $F$ is contained in a $\sigma$-stable proper parabolic subgroup of $G$.
\end{prop}

\noindent For if (a) does not hold then $F\leq P_{\rm opt}(F)$, which is a proper parabolic subgroup of $G$.   The idea is to show that $P_{\rm opt}(F)$ is $\sigma$-stable (there are some complications when $(G,p)= (B_2,2)$, $(F_4,2)$ or $(G_2,3)$).

Proposition~\ref{prop:Gcr_not_Gcr} yields a bound on the number of maximal subgroups of simple groups of Lie type \cite[Thm.\ 1.2]{LMS}:

\begin{thm}
 Let $N, R\in {\mathbb N}$ be positive integers, and let $\Gamma$ be an almost simple group whose socle is a finite simple group of Lie type of rank at most $R$. Then the number of conjugacy classes of maximal subgroups of order at most $N$ in $\Gamma$ is bounded by a function $f(N, R)$ of $N$ and $R$ only.
\end{thm}

\noindent An important ingredient in the proof is the following observation.  If $F$ is a finite group of order $N$ then of course the number $n= n(F,\sigma)$ of $G^\sigma$-conjugacy classes of embeddings $\rho\colon F\ra G^\sigma$ is finite, since $G^\sigma$ is finite.  If we consider embeddings $\rho$ such that $\rho(F)$ is $G$-cr, however, then Theorem~\ref{thm:fin_cr_rep} together with Lang's Theorem shows there is a bound for $n$ that does not depend on $\sigma$.

\subsection{Geometric invariant theory}

As we have seen, geometric invariant theory is an important tool for proving results about $G$-complete reducibility.  Now we give some results from geometric invariant theory which can be proved using ideas from $G$-complete reducibility.

We start by giving a rigidity result for $G$-cr subgroups.

\begin{prop}[{\cite[Lem.\ 4.1]{martin_genred}}]
\label{prop:rigidity}
 The group $G$ has only countably many conjugacy classes of $G$-cr subgroups.
\end{prop}

\noindent Suppose we are given a conjugation-stable family ${\mathcal F}$ of subgroups of $G$ that is parametrised algebraically: for instance, if $X$ is a $G$-variety then we can take ${\mathcal F}$ to be the family of stabilisers $G_x$ as $x$ varies over the points of $X$.  Suppose moreover that every $H\in {\mathcal F}$ is $G$-cr.  Then ${\mathcal F}$ contains only finitely many conjugacy classes of subgroups.  The proof has a model-theoretic flavour.  Since ${\mathcal F}$ is parametrised algebraically, it is given by first-order conditions, so without loss we can extend scalars and assume the algebraically closed field $k$ is uncountable.  Then ${\mathcal F}$ has to contain either only finitely many, or uncountably many, conjugacy classes of subgroups, essentially because a variety over $k$ is either finite or uncountable.  But by Proposition~\ref{prop:rigidity}, ${\mathcal F}$ can contain at most countably many conjugacy classes of subgroups, so it contains only finitely many.

\begin{cor}[{\cite[Cor.\ 1.5]{martin_genred}}]
\label{cor:princ_stab}
 Let $X$ be a quasi-projective $G$-variety, and suppose $X$ has an open dense set $U_1$ such that $G_x$ is $G$-cr for all $x\in U_1$.  Then $X$ has an open dense subset $U_2$ such that the stabilisers $G_x$ for $x\in U_2$ are all conjugate to each other.
\end{cor}

\noindent Corollary~\ref{cor:princ_stab} fails without the assumption that the $G_x$ are $G$-cr: see \cite[Ex.\ 8.3]{martin_genred}.

We finish with one further GIT-theoretic result.  Let $X$ be a $G$-variety and let $H$ be a $G$-cr subgroup of $X$.  Then the fixed point set $X^H$ is stabilised by $N_G(H)$, and the inclusion $X^H\subseteq X$ gives rise to a morphism $\psi\colon X^H/\!\!/N_G(H)\ra X/\!\!/G$ of quotient varieties (note that $N_G(H)$ is reductive, by Section~\ref{sec:Gcr_GIT}). 

\begin{prop}[{\cite[Thm.\ 1.1]{BGM}}]
 The morphism $\psi$ is finite.
\end{prop}



\begin{thebibliography}{00}

\bibitem{BHMR_nonconn}
M.~Bate, S.~Herpel, B.~Martin, G.~R\"ohrle,
\emph{$G$-complete reducibility in non-connected groups},
Proc.\ Amer.\ Math.\ Soc.\ \textbf{143} (2015), no.\ 3, 1085--1100.

\bibitem{BGM}
M.~Bate, B.~Martin, H.~Geranios,
\emph{Orbit closures and invariants}, Math. Z. \textbf{293}, no. 3-4, (2019), 1121--1159. 

\bibitem{BHMR_ss}
M.~Bate, S.~Herpel, B.~Martin, G.~R\"ohrle,
\emph{$G$-complete reducibility and semisimple modules},
Bull.\ London Math.\ Soc.\ \textbf{43} (2011), no.\ 6, 1069--1078.

\bibitem{BMR}
M.~Bate, B.~Martin, G.~R\"ohrle,
\emph{A geometric approach to complete reducibility},
Invent.\ Math. \textbf{161} (2005), no.\ 1, 177--218.

\bibitem{BMR_commuting}
M.~Bate, B.~Martin, G.~R\"ohrle,
\emph{Complete reducibility and commuting subgroups},
J. Reine Angew.\ Math.\ \textbf{621} (2008), 213--235.

\bibitem{BMR_sTCC}
M.~Bate, B.~Martin, G.~R\"ohrle,
\emph{The strong {C}entre {C}onjecture: an invariant theory approach},
J. Algebra \textbf{372} (2012), 505--530.

\bibitem{BMR_kuls}
M.~Bate, B.~Martin, G.~R\"ohrle,
\emph{On a question of {K}\"ulshammer for representations of finite groups in reductive groups},
Israel J.\ Math.\ \textbf{214} (2016), 463--470.

\bibitem{BMRT}
M.~Bate, B.~Martin, G.~R\"ohrle, R.~Tange,
\emph{Complete reducibility and separability},
Trans.\ Amer.\ Math.\ Soc., \textbf{362} (2010), no.\ 8, 4283--4311.

\bibitem{BMRT_git}
M.~Bate, B.~Martin, G.~R\"ohrle, R.~Tange,
\emph{Closed orbits and uniform $S$-instability in geometric invariant theory},
Trans.\ Amer.\ Math.\ Soc.\ \textbf{365} (2013), no.\ 7, 3643--3673.

\bibitem{borel} 
A.~Borel, \emph{Linear Algebraic Groups}, Graduate Texts in 
Mathematics, \textbf{126}, Springer-Verlag 1991.

\bibitem{cram}
G.-M.~Cram,
\emph{On a question of K\"ulshammer about algebraic group actions: an example}, appendix to \cite{slodowy},
Austral.\ Math.\ Soc.\ Lect.\ Ser.\ \textbf{9},
Algebraic groups and Lie groups, 346--348,
Cambridge Univ.\ Press, Cambridge, 1997.

\bibitem{herpel}
S.~Herpel,
\emph{On the smoothness of centralizers in reductive groups}, Trans.\ Amer.\ Math.\ Soc.\ \textbf{365} (2013), no.\ 7, 3753--3774.

\bibitem{hesselink}
W.H.~Hesselink,
\emph{Uniform instability in reductive groups},
J. Reine Angew.\ Math.\ \textbf{303/304} (1978), 74--96.

\bibitem{hum}
J.E.~Humphreys,
\emph{Linear Algebraic Groups},
Springer-Verlag, New York, 1975.

\bibitem{kempf}
G.R.~Kempf,
\emph{Instability in invariant theory},
Ann.\ Math.\ \textbf{108} (1978), no.\ 2, 299--316.

\bibitem{kuls}
 B.~K\"ulshammer,
 \emph{{D}onovan's conjecture, crossed products and algebraic group actions},
 Israel J.\ Math.\, \textbf{92} (1995), no.\ 1--3, 295--306.

\bibitem{liebeckseitz0}
M.W.~Liebeck, G.M.~Seitz,
\emph{Reductive subgroups of exceptional algebraic groups},
Mem.\ Amer.\ Math.\ Soc.\ no.\ \textbf{580} (1996).

\bibitem{LMS}
M.W.~Liebeck, B.M.S.~Martin, A.~Shalev,
\emph{On conjugacy classes of maximal subgroups of finite simple groups, and a related zeta function},
Duke Math.\ J. \textbf{128} (2005), no.\ 3, 541--557.

\bibitem{martin}
B.M.S.~Martin,
\emph{Reductive subgroups of reductive groups in nonzero characteristic},
J. Algebra \textbf{262} (2003), no.\ 2, 265--286.

\bibitem{martin_genred}
B.M.S.~Martin,
\emph{Generic stabilisers for actions of reductive groups},
Pacific J. Math.\ \textbf{279} (2015), no.\ 1--2, 397--422.

\bibitem{newstead}
P. E.~Newstead,
\emph{Introduction to moduli problems and orbit spaces},
Tata Institute of Fundamental Research Lectures on Mathematics and Physics \textbf{51}.
Tata Institute of Fundamental Research, Bombay, 1978.

\bibitem{rich}
R.W.~Richardson,
\emph{Affine coset spaces of reductive algebraic groups},
Bull.\ London Math.\ Soc.\ \textbf{9} (1977), no.\ 1, 38--41.

\bibitem{rich4}
R.W.~Richardson,
\emph{Conjugacy classes of $n$-tuples in {L}ie algebras and algebraic groups},
Duke Math.\ J. \textbf{57}, (1988), no.\ 1, 1--35.

\bibitem{serre1}
J.-P. Serre,
\emph{The notion of complete reducibility in group theory},
Moursund Lectures, Part II, University of Oregon, 1998,\
{\tt arXiv:math/0305257v1 [math.GR]}.

\bibitem{rousseau}
G.~Rousseau,
\emph{Immeubles sph\'eriques et th\'eorie des invariants},
C.R.A.S. \textbf{286} (1978), 247--250.

\bibitem{serre2}
J-P.~Serre,
\emph{Compl\`ete r\'eductibilit\'e},
S\'eminaire Bourbaki, 56\`eme ann\'ee, 2003--2004, n$^{\rm o}$ 932.

\bibitem{serre3}
J.-P.~Serre,
\emph{La notion de compl\'ete r\'eductibilit\'e dans les immeubles sph\'eriques et les groupes r\'eductifs},
S\'eminaire au Coll\'ege de France, r\'esum\'e dans [\cite{tits}, pp.\ 93--98](1997).

\bibitem{slodowy}
P.~Slodowy,
\emph{Two notes on a finiteness problem in the representation theory
of finite groups},
Austral.\ Math.\ Soc.\ Lect.\ Ser.\ \textbf{9},
Algebraic groups and Lie groups, 331--348,
Cambridge Univ.\ Press, Cambridge, 1997.

\bibitem{springer} 
T.A. ~Springer, \emph{Linear Algebraic Groups},
Second edition. Progress in Mathematics, 9. Birkh\"auser Boston, Inc., 
Boston, MA, 1998.

\bibitem{stewartG2}
 D.I.~Stewart,
 \emph{The reductive subgroups of $G_2$},
 J. Group Theory \textbf{13} (2010), no.\ 1, 117--130.
   
\bibitem{stewartF4}
   D.I.~Stewart,
   \emph{The reductive subgroups of $F_4$},
   Mem.\ Amer.\ Math.\ Soc.\ \textbf{223} (2013), no.\ 1049, vi+88pp.
   
\bibitem{stewartTAMS}
 D.I.~Stewart,
 \emph{On unipotent algebraic $G$-groups and 1-cohomology},
 Trans.\ Amer.\ Math.\ Soc.\ \textbf{365} (2013), no.\ 12, 6343--6365.

\bibitem{tits}
J.~Tits,
\emph{Th\'eorie des groupes}, R\'esum\'e des Cours et Travaux, Annuaire du Coll\`ege de France, 97e ann\'ee (1996--1997), 89--102.
 
\bibitem{uchiyama2}
T.~Uchiyama,
\emph{Separability and complete reducibility of subgroups of the Weyl group of a simple algebraic group of type $E_7$}, J. Algebra \textbf{422} (2014), 357--372. 

\bibitem{uchiyama3}
T.~Uchiyama,
\emph{Non-separability and complete reducibility: $E_n$ examples with an application to a question of K\"ulshammer},
J.\ Group Theory \textbf{20} (2017), no.\ 5, 925--944.

\bibitem{uchiyama}
T.~Uchiyama,
\emph{Complete reducibility of subgroups of reductive algebraic
groups over nonperfect fields II},
Comm.\ Algebra \textbf{45} (2017), no.\ 11, 4833--4845.

\bibitem{vinberg}
E.B.~Vinberg,
\emph{On invariants of a set of matrices},
J. Lie Theory \textbf{6} (1996), 249--269.
    
\end{thebibliography}

\end{document}